\documentclass{siamart190516}
\usepackage{amsmath}
\usepackage{graphicx}
\usepackage{amsfonts}
\usepackage{amssymb}
\usepackage{float}
\usepackage{color}
\usepackage{algorithm}
\usepackage{algorithmic}
\usepackage{hyperref}
\usepackage{savesym}
\usepackage{datetime}
\usepackage[normalem]{ulem}

\usepackage{multirow}
\usepackage{framed} 
\usepackage{xcolor} 
\colorlet{shadecolor}{gray!20} 
\definecolor{mygreen}{rgb}{0,0.6,0}

\newtheorem{remark}[theorem]{Remark}

\def\sym{\mathrm{sym}}

\def\tr{\mathrm{tr}}
\def\i{\mathrm{i}}
\def\Rbkk{{\mathbb{R}^{2k\times 2k}}}
\def\Rbnk{{\mathbb{R}^{2n\times 2k}}}
\def\Rbnn{{\mathbb{R}^{2n\times 2n}}}
\def\Spkn{{\mathrm{Sp}(2k,2n)}}
\def\Spn{{\mathrm{Sp}(2n)}}
\def\Spk{{\mathrm{Sp}(2k)}}
\def\SPD{\mathcal{SPD}}
\def\diag{\mathrm{diag}}
\def\OrS{\mathrm{OrSp}}
\def\calI{\mathcal{I}}

\def\sp{\mathrm{span}}
\def\R{\mathbb{R}}
\newcommand{\skewset}{{\cal S}_{\mathrm{skew}}}
\newcommand{\symset}{{\cal S}_{\mathrm{sym}}}

\hypersetup{colorlinks=true,linkcolor=blue,citecolor=blue}    

\newcommand{\TX}{{\mathrm{T}_{X}}\Spkn}
\newcommand{\rgrad}[1]{\mathrm{grad}_\rho f(#1)}
\newcommand{\abs}[1]{\left|#1\right|}
\newcommand{\dkh}[1]{\left(#1\right)}

\newcommand{\jkh}[1]{\left\langle#1\right\rangle}

\title{Symplectic eigenvalue problem via Trace minimization and Riemannian optimization\thanks{This work was supported by the Fonds de la Recherche Scientifique – FNRS and the Fonds Wetenschappelijk Onderzoek – Vlaanderen under EOS Project no. 30468160. It was finished during a visit of the first author to Vietnam Institute for Advanced Study in Mathematics (VIASM) whose support was gratefully acknowledged.} }

\author{Nguyen Thanh Son\thanks{Department of Mathematics and Informatics, Thai Nguyen University of Sciences, 24118 Thai Nguyen, Vietnam (ntson@tnus.edu.vn).}
	\and 
	P.-A. Absil\thanks{ICTEAM Institute, UCLouvain, 1348 Louvain-la-Neuve, Belgium (pa.absil@uclouvain.be, gaobin@lsec.cc.ac.cn).}
	\and Bin Gao\footnotemark[3]
	\and Tatjana Stykel\thanks{Institute of Mathematics, University of Augsburg, 86159 Augsburg,  Germany (stykel@math.uni-augsburg.de).}	
}

\begin{document}
\maketitle

\begin{abstract} 	
We address the problem of computing the smallest 
symplectic eigenvalues and the corresponding eigenvectors of symmetric positive-definite matrices in  
the sense of Williamson's theorem.
It is formulated as minimizing a~trace cost function 
over the symplectic Stiefel manifold.
We first investigate various theoretical aspects of this optimization problem such as characterizing the sets of critical points, saddle points, and global minimizers as well as proving that non-global local minimizers do not exist. Based on our recent results on constructing Riemannian structures on the symplectic Stiefel manifold and the associated optimization algorithms, we then propose solving the symplectic eigenvalue problem in the framework of Riemannian optimization. Moreover, 
a~connection of the sought solution with the eigenvalues of a special class of Hamiltonian matrices is discussed. Numerical examples are presented.
\end{abstract}



\pagestyle{myheadings}
\thispagestyle{plain}
\markboth{N.T. SON, P.-A. ABSIL, B. GAO, AND T. STYKEL}{SYMPLECTIC EVP VIA RIEMANNIAN TRACE MINIMIZATION}

\section{Introduction}\label{Sec:Intro}
Given a~positive integer $n$, let us consider the matrix
\begin{equation*}
J_{2n} = \begin{bmatrix}0&I_n\\-I_n&0
\end{bmatrix} \in \Rbnn,
\end{equation*}
where $I_n$ denotes the $n\times n$ identity matrix.
A matrix $X\in \Rbnk$ with $k\leq n$ is said to be {\em symplectic} if it holds $X^TJ_{2n}X = J_{2k}$. Although the term ``symplectic" previously seemed to apply to square matrices only, it has recently been used for rectangular ones as well \cite{PengM16,GSAS20}. Note that $J_{2n}$ is orthogonal, skew-symmetric, symplectic, and sometimes referred to as the Poisson matrix  \cite{PengM16}. Symplectic matrices appear in a~variety of applications including quantum mechanics \cite{deGo06}, Hamiltonian dynamics \cite{HofeZ11,VanDSJ14}, systems and control theory  \cite{Fran87, HinrS91,LancR95} and optimization problems \cite{Fior16,BirtCC20}. 
The set of all symplectic matrices is denoted by $\Spkn$. When $k = n$, we write $\Spn$ instead of $\Spkn$. These matrix sets have a~rich geometry structure: $\Spkn$ is a~Riemannian manifold \cite{GSAS20}, also known as the symplectic Stiefel manifold, whereas $\Spn$ forms additionally 
a~noncompact Lie group \cite[Lemma~1.15]{Fome95}.

There are fundamental differences between symplectic and orthonormal matrices: notably, $\Spkn$ is unbounded \cite{GSAS20}. However, their definitions look alike (replacing $J$ by $I$ in the definition of symplectic matrices yields that of orthonormal ones) and several properties of orthonormal matrices have their counterparts for symplectic matrices, e.g., they have full rank and they form a submanifold. Of interest here is the diagonalization of symmetric positive-definite (spd) matrices. The fact that every spd matrix can be reduced by an orthogonal congruence to a diagonal matrix with positive diagonal elements 
is well-known and can be found in any standard linear algebra textbook. This problem is also called the eigenvalue decomposition as the diagonal entries of the diagonalized matrix are the eigenvalues of the given one. Its symplectic counterpart is known as Williamson's theorem~\cite{Will36} which states that for any  spd matrix $M\in\mathbb{R}^{2n\times 2n}$, there exists 
$S\in\Spn$ such that 
\begin{equation}\label{eq:sympleigdecomp}
S^TMS = \begin{bmatrix}
D & 0 \\ 0 & D
\end{bmatrix},
\end{equation}
where $D=\diag(d_1,\ldots,d_n)$ with positive diagonal elements.
This decomposition is referred to as 
Williamson's diagonal form or Williamson's normal form 
of $M$.  The values $d_i$ are called the {\em symplectic eigenvalues} of $M$, and the columns of $S$ form a~{\em symplectic eigenbasis} in 
$\mathbb{R}^{2n}$. Constructive proofs of Williamson's theorem can be found in \cite{SimoCS99,Part13,Ikra18}. Symplectic eigenvalues have wide applications in quantum mechanics and optics; they are important quantities to characterize quantum systems and their subsystems with Gaussian states \cite{Hiro06,Part13,KrbeTV14}. Especially, in the Gaussian marginal problem, knowledge on symplectic eigenvalues helps to determine local entropies which are compatible with a given joint state \cite{EiseTRS08}. 

The computation of standard eigenvalues is a well-established subfield in numeri\-cal linear algebra, see, e.g., \cite{Kres05,Watk07,Saad11} and many other textbooks related to matrix analysis and computations. Particularly, numerical methods based on optimization were extensively studied 
where either a matrix trace or Rayleigh quotient is minimized with some constraints. 
The generalized eigenvalue problems (EVPs) 
were investigated in \cite{SameW82,KovaV95,SameT00,LianLB13} using trace minimization. This approach was also applied to a special class of Hamiltonian matrices in the context of (generalized) linear response EVP \cite{BaiL12,BaiL13,BaiL14}. 
The authors of \cite{EdelAS98,AbsiBGS05,AbsiBG06,BakeAG06} approached the Rayleigh quotient or trace minimization problem by using Riemannian optimization on an appropriately chosen matrix manifold \cite{AbsiMS08} such as the Stiefel manifold and the Grassmann manifold. 
However, only very few works devoted to computing symplectic eigenvalues can be found in the literature. In addition to some constructive proofs, e.g., \cite{SimoCS99,Part13}, 
which lead to numerical methods suitable 
for small to medium-sized problems only, the approaches in \cite{Amod06,Ikra18} are based on the one-to-one correspondence between spd matrices and a~special class of Hamiltonian ones, the so-called positive-definite Hamiltonian (pdH) matrices. 
Specifically, it was proposed in \cite{Ikra18} to compute the symplectic eigenvalues of $M$ by transforming the pdH matrix $J_{2n}M$ into a normal form by using elementary symplectic transformations as described in \cite{Ikra91}. Furthermore,  the symplectic Lanczos method for 
computing several extreme eigenvalues of pdH matrices developed in \cite{Amod03,Amod06} was also based on 
a~similar relation. Perturbation bounds for Williamson's diagonal form were presented in \cite{IdelGW17}.

To the best of our knowledge, there is no algorithmic work that relates the computation of symplectic eigenvalues to the  optimization framework similar to that for the standard EVP. In \cite{Hiro06,BhatJ15}, a~connection 
between the sum of the $k$ smallest symplectic eigenvalues of an~spd matrix and the minimal trace of a matrix function defined on the set of symplectic matrices was established. Note that computation was not the focus and no algorithms were discussed in these works. Moreover, no practical procedure can be directly implied from the relation.

In this paper, building on results of~\cite{Hiro06,BhatJ15} and on various additional properties of the trace minimization problem, 
we construct an algorithm to compute the smallest symplectic eigenvalues via solving an optimization problem with symplectic constraints by exploiting the Riemannian structure of $\Spkn$ investigated recently in \cite{GSAS20}. Our goal is not merely to find a way to minimize the trace cost function, but also to investigate the intrinsic connection between the symplectic EVP and the trace minimization problem. To this end, our contributions are mainly reflected in the following aspects. (i)~We characterize the set of eigenbasis matrices in Williamson's diagonal form of an spd matrix (Theorem~\ref{prop:trackingS}) as well as the sets of critical points (Theorem~\ref{prop:Critical_Orthosympl} and Corollary~\ref{corol:CriticalClasses}), saddle points (Proposition~\ref{prop:SaddlePoint}) and the minimizers (Theorem~\ref{prop:minimizers_Orthosympl} and Corollary~\ref{corol:MimimizersClasses}) of the associated trace minimization problem and prove the non-existence of non-global local minimizers (Proposition~\ref{prop:localnonglobal}). 
Some of these findings turn out to be important extensions of the existing results for the standard EVP. (ii) Based on a~recent development on symplectic optimization derived in \cite{GSAS20}, we propose an algorithm (Algorithm~\ref{alg:SymplEigProb}) to solve the symplectic EVP via Riemannian optimization. (iii) As an application, we consider computing the standard eigenvalues and the corresponding eigenvectors of the associated pdH matrix. Numerical examples are reported to verify the effectiveness of the proposed algorithm. 

To avoid ambiguity, we would like to mention that the term ``symplectic eigenvalue problem" or ``symplectic eigenproblem" was also used in some works, e.g., \cite{BunsM86,BennF98,Fass02}, 
in a~different meaning. There, symplectic matrices are used as a tool to compute standard eigenvalues of structured matrices such as Hamiltonian, skew-Hamiltonian, and symplectic matrices. The motivation behind this is that symplectic similarity transformations  preserve these special structures. The resulting structure-preserving methods are, therefore, referred to as symplectic methods. Here, we focus instead on the computation of the symplectic eigenvalues of spd matrices, where symplectic matrices are involved due to Williamson's diagonal form \eqref{eq:sympleigdecomp}, and a~special Hamiltonian EVP is considered as an~application only.

The rest of the paper is organized as follows. In section~\ref{sec:NotationPreliminary}, we 
introduce the notation and review some basic facts for structured matrices.  
In section~\ref{sec:WilliasonTheo}, we define the symplectic EVP, 
revisit Williamson's theorem on  diagonalization of spd matrices, and 
characterize the set of symplectically diagonalizing matrices.
We also establish a~relation between the standard and symplectic eigenvalues for spd and skew-Hamiltonian matrices. 
In section~\ref{sec:SymplTraceMin}, we go deeply into the symplectic trace minimization problem and 
study the connection between the symplectic EVP and trace minimization. 
In section~\ref{sec:EigenCompviaRiemOpt}, we present a~Riemannian optimization algorithm for computing the smallest symplectic eigenvalues as well as the corresponding eigenvectors. 
Additionally, we discuss the computation of standard eigenvalues of pdH matrices.
Some numerical results are given in section~\ref{Sec:NumerExam}. Finally, the conclusion is provided in section~\ref{Sec:Concl}.

\section{Notation and preliminaries}\label{sec:NotationPreliminary}

In this section, after stating some conventions for notation, we introduce several structured matrices used in this paper and collect their useful properties. 

In the Euclidean space $\R^{2n}$, $e_i$ denotes the $i$-th canonical basis vector for 
\linebreak\mbox{$i=1,\ldots,2n$}. 
The Euclidean inner product of two matrices $X, Y\in \R^{n\times m}$ is denoted by $\langle X,Y\rangle:=\tr(X^TY)$, where $\tr(\cdot)$ is the trace operator and $X^T$ stands for the transpose of $X$. 
Given $A\in \R^{m\times m}$, 
$\sym(A):=\frac{1}{2}(A+A^T)$ denotes the symmetric part of~$A$. 
We let $\diag(a_1,\ldots,a_m)\in\R^{m\times m}$ denote the diagonal matrix with the components $a_1,\ldots,a_m$ on the diagonal. 
This notation is also used for block diagonal matrices, where each $a_i$ is a submatrix block. We use $\sp(A)$ to express the subspace spanned by the columns of $A$. 
Furthermore, $\symset(n)$, $\SPD(n)$, and  $\skewset(n)$ denote the sets of all symmetric, symmetric positive-definite, and skew-symmetric $n\times n$ mat\-rices, respectively.
For a~twice continuously differentiable function $f:\R^{n\times m} \to \R$, we denote by $\nabla f(X)$ and $\nabla^2 f(X)$, respectively, the Euclidean gradient and the Hessian of $f$ at~$X$. Moreover, $\mathrm{D}h(X)$ stands for the Fr\'{e}chet derivative at $X$ of a mapping $h$ between Banach spaces, if it exists.

A~matrix $H\in\R^{2n\times 2n}$ is called \emph{Hamiltonian} if $(J_{2n}^TH)^T=J_{2n}^TH$.
It is well-known, e.g., \cite{VanL84}, that the eigenvalues of such a~matrix appear in pairs $(\lambda,-\lambda)$, if $\lambda\in\mathbb{R}\cup {\rm i}\mathbb{R}$, or in quadruples $(\lambda, -\lambda,\overline{\lambda},-\overline{\lambda})$, if $\lambda\in\mathbb{C}\setminus\{\mathbb{R}\cup {\rm i}\mathbb{R}\}$.
Here, ${\rm i}=\sqrt{-1}$ denotes the imaginary unit. Further, a Hamiltonian matrix $H \in\R^{2n\times 2n}$ 
is called \emph{positive-definite Hamiltonian} (pdH) if its symmetric generator $J_{2n}^TH$ is positive definite. The eigenvalues of the pdH matrix are purely imaginary \cite{AmodIT05}.

A~matrix $N\in\mathbb{R}^{2n\times 2n}$ is called \emph{skew-Hamiltonian} if $(J_{2n}^TN)^T=-J_{2n}^TN$. Each eigenvalue of $N$ has even algebraic multiplicity. Skew-Hamiltonian matrices play an~important role in the computation of eigenvalues and invariant subspaces of Hamiltonian matrices, see \cite{BennKM05} for a~survey.

A~matrix $K\in\mathbb{R}^{2n\times 2n}$ is called \emph{orthosymplectic}, if it is both orthogonal and symplectic, i.e., $K^TK=I_{2n}$ and $K^TJ_{2n}K=J_{2n}$. We denote the set of $2n\times 2n$ orthosymplectic matrices by $\OrS(2n)$. It is well-known that similarity transformations of Hamiltonian, skew-Hamiltonian and symplectic
matrices with (ortho)symplectic matrices preserve the corresponding matrix structure. This property is often used in structure-preserving algorithms for solving structured EVPs, e.g., \cite{VanL84,Fass02,BennKM05}.

Next, we present some useful facts on symplectic and orthosymplectic matrices which will be exploited later. 

\begin{proposition}\label{prop:SymplGroup}
	\begin{itemize}
		\item[i)] Let $S \in\Spn$. Then $ S^{-1}, S^T \in \Spn$.
		\item[ii)] The set of orthosymplectic matrices $\OrS(2n)$ is a group characterized by 
		\begin{equation*}
		\OrS(2n) = \left\{K = \begin{bmatrix}
		K_1&K_2\\-K_2&K_1
		\end{bmatrix}\;:\; K_1^TK_2 = K_2^TK_1,\ K_1^TK_1 + K_2^TK_2 = I \right\}.
		\end{equation*}
		\item[iii)] For $S, T \in \Spkn$, $\mathrm{span}(S)=\mathrm{span}(T)$ if and only if there exists 
		a~matrix $K \in \Spk$ such that $T = SK$.
	\end{itemize}
\end{proposition}

\begin{proof}
	i) These facts have been proved
	in various sources, e.g., \cite[Section~2]{Hiro06} or \cite[Proposition~2 in Chapter~1]{HofeZ11}. 
	
	ii) The representation for elements of $\OrS(2n)$ has been proved in \cite[Section~2.1.2]{deGo06} or \cite[Section~7.8.1]{GoluV13}. 
	This set is a group because it is 
	 the intersection of two groups with the same operation and identity element.
	
	iii) 
	 If $k=n$, the proof is straightforward since $\Spn$ is a group. Otherwise, the sufficiency 
	 	immediately follows from the relation $T = SK$. To prove the necessity, we assume that $S, T \in \Spkn$
		with $\mathrm{span}(S)=\mathrm{span}(T)$. Then there exists a~nonsingular matrix $K \in \mathbb{R}^{2k \times 2k}$ such that $T = SK$. 
	The simplecticity of $K$ is verified by
	$K^TJ_{2k}K = K^TS^TJ_{2n}SK = T^TJ_{2n}T=J_{2k}$.
\end{proof}

\section{Williamson's theorem revisited} \label{sec:WilliasonTheo} 
In this section, we discuss Williamson's theorem and related issues in detail. This includes a definition of symplectic eigenvectors, a~characterization of symplectically diagonalizing matrices, and 
the methods for computing Wil\-liam\-son's diagonal form for general spd matrices and for spd and skew-Hamiltonian matrices.

\subsection{Williamson's diagonal form and symplectic eigenvectors}
\label{subsec:SymplEig}

First, we review some facts related to Williamson's theorem. Let a~matrix $M\in\SPD(2n)$ be transformed into Williamson's diagonal form \eqref{eq:sympleigdecomp} 
with a symplectic transformation matrix $S=[s_1,\ldots,s_n,s_{n+1},\ldots,s_{2n}]$ and a~diagonal matrix $D=\mbox{diag}(d_1,\ldots,d_n)$ with the symplectic eigenvalues on the diagonal 
in the non-decreasing order, i.e., \mbox{$0<d_1\leq \ldots\leq d_n$}.
In this case, we will say that $S$ \emph{symplectically diagonalizes} $M$ or that $S$ is a \emph{symplectically diagonalizing} matrix, when $M$ is clear from the context. 
Note that the set of symplectic eigenvalues,  
also called the \emph{symplectic spectrum}  of~$M$, is known to be unique \cite[Theorem~8.11]{deGo06}, while the symplectically diagonalizing matrix $S$ is not unique. It has been shown in 
\cite[Proposition~8.12]{deGo06} that if 
$S$ and $T$ symplectically diagonalize $M$, then $S^{-1}T\in\OrS(2n)$. 

The \emph{multiplicity} of the symplectic eigenvalue $d_j$, $j=1,\ldots,n$, is the number of times it is repeated in $D$.
Note that this definition differs from that for standard eigenvalues, where the appearance of the eigenvalue in $\diag(D,D)$ is counted. The reasons for this discrepancy will get cleared after introducing symplectic eigenvectors, see, e.g., \cite{BhatJ15,JainM20} and the references therein.

A pair of vectors $(u,v)$ in $\R^{2n}$ is called (\emph{symplectically}) \emph{normalized} if $\langle u,J_{2n}v\rangle = 1$. Two pairs of vectors $(u_1, v_1)$ and $(u_2, v_2)$ are said to be \emph{symplectically orthogonal} if
$$
\langle u_i,J_{2n}v_j\rangle = \langle u_i,J_{2n}u_j\rangle = \langle v_i,J_{2n}v_j\rangle =0
\enskip \mbox{ for } i\not= j,\; i, j = 1,2.
$$
A matrix $X=[u_1, \ldots, u_k, v_1, \ldots, v_k] \in\R^{2n\times 2k}$ is said to be \emph{normalized} if each pair $(u_i, v_i)$, $i=1,\ldots, k,$ is normalized. 
It is called \emph{symplectically orthogonal}  
if the pairs of vectors $(u_i, v_i)$ are mutually symplectically orthogonal.  
Note that the symplecticity of $X$ is equivalent to the fact that $X$ is normalized and symplectically orthogonal.
For $k=n$, a~normalized and symplectically orthogonal vector set forms a~\emph{symplectic basis} in $\mathbb{R}^{2n}$. 

The two columns of a~matrix $X \in \mathbb{R}^{2n\times 2}$ are called a 
{\em symplectic eigenvector pair} of $M\in \SPD(2n)$ associated with a~symplectic eigenvalue $\lambda$ if it holds
\begin{equation}\label{eq:eigenvector_pair}
MX = J_{2n}X\begin{bmatrix}
0 & -\lambda \\ \lambda & \enskip 0
\end{bmatrix}.
\end{equation}
If $X$ is additionally symplectic, we call its columns a~\emph{normalized symplectic eigenvector pair}. Since each symplectic eigenvalue always needs a~pair of symplectic eigenvectors to define, this explains the above definition of the multiplicity. 

More general, the columns of 
$X \in  \mathbb{R}^{2n\times 2k}$ are called a {\em symplectic eigenvector set} of $M\in \SPD(2n)$ associated with the symplectic eigenvalues $\lambda_1,\ldots,\lambda_k$, if it holds
\begin{equation}\label{eq:SymplEigvectorSet_2}
MX = J_{2n}X\begin{bmatrix}
0&-\Lambda \\ \Lambda& \enskip 0
\end{bmatrix}
\end{equation}
with $\Lambda =\mathrm{diag}(\lambda_1,\ldots,\lambda_k)$. If $X$ is, in addition, symplectic, we say that its columns form a~\emph{normalized symplectic eigenvector set.} 

\begin{remark}
If $X\in \Spkn$ satisfies \eqref{eq:SymplEigvectorSet_2}, then due to the uniqueness 
of the symplectic eigenvalues (conventionally arranged in non-decreasing order), there always exists 
a~strictly increasingly ordered index set $\calI_k=\{i_1, \ldots,i_k\}\subset\{1,\ldots,n\}$ such that \mbox{$\Lambda=\diag(d_{i_1},\ldots,d_{i_k})$}. Therefore, in this paper, we will use $X_{\calI_k}$ to denote any normalized symplectic eigenvector set associated with the symplectic eigenvalues $d_{i_1},\ldots,d_{i_k}$. If $\calI_k=\{1,\ldots, k\}$, we will write $X_{1:k}$.
\end{remark} 

Multiplying both sides of Williamson's diagonal form \eqref{eq:sympleigdecomp} from the left with $S^{-T} = J_{2n}^{}SJ_{2n}^T$, we obtain
\begin{equation}\label{eq:sympleigdef}
MS = J_{2n}^{}SJ_{2n}^T\begin{bmatrix}
D & 0 \\ 0 & D
\end{bmatrix} = J_{2n}S\begin{bmatrix}
0 & -D \\ D & \enskip 0
\end{bmatrix}.
\end{equation}
This implies that for any ordered index set 
$\calI_k=\{i_1, \ldots,i_k\}\subset\{1,\ldots,n\}$,
the columns of the symplectic submatrix $[s_{i_1},\ldots,s_{i_k},s_{n+i_1},\ldots,s_{n+i_k}]$ of $S$ form a~normalized symplectic eigenvector set of $M$ associated with $d_{i_1},\ldots,d_{i_k}$.
Note that $[cs_i, \  cs_{n+i}]$ with $c\not\in \{-1,0,1\}$ is a~symplectic eigenvector pair associated with $d_i$ but not normalized. 

Taking into account \eqref{eq:sympleigdef}, Williamson's theorem can alternatively be restated as follows: For any $M\in\SPD(2n)$, there exists a~normalized symplectic eigenvector set of $M$ that constitutes a symplectic basis in~$\mathbb{R}^{2n}$.

Next, we collect some useful facts on symplectic eigenvectors.
\begin{proposition}\label{prop:Jain2020}\textup{\cite[Corollaries~2.4 and~5.3]{JainM20}}
	Let $M\in \SPD(2n)$. 
		\begin{itemize}
		\item[i)] Any two symplectic eigenvector pairs corresponding
		to two distinct symplectic eigenvalues of $M$ are symplectically orthogonal. 
		\item[ii)] Let $\lambda$ be a symplectic eigenvalue of $M$ of multiplicity $m$ and let the columns of $X\in \mathbb{R}^{2n\times 2m}$ be a normalized symplectic eigenvector set associated with $\lambda$. Then the columns of a~matrix $Y\in \mathbb{R}^{2n\times 2m}$ form also a~normalized symplectic eigenvector set associated with $\lambda$ if and only if there exists $K\in \OrS(2m)$ such that $Y = XK$.
	\end{itemize}
\end{proposition}

We conclude this subsection by mentioning a~connection between the symplectic eigenvalues and 
eigenvectors of the spd matrix $M$ and the standard eigenvalues and eigenvectors of the pdH matrix $J_{2n}M$. This result is not new and has already been established in a~slightly different form in \cite[Theorem~8.11]{deGo06} and \cite[Lemma~2.2]{JainM20}.

\begin{proposition}\label{prop:SEVandEV}
	Let $M \in \SPD(2n)$ and let $S = [s_1,\ldots,s_{2n}]$ be a~symplectically diagonalizing matrix of $M$. Then $d_j, j=1,\ldots,n$, are the symplectic eigenvalues of $M$ if and only if $\pm \i d_j, j=1,\ldots,n,$ are the standard eigenvalues of the pdH matrix $H=J_{2n}M$. Moreover,  for any $j=1,\ldots,n$, $s_j \pm \i s_{n+j}$ is an eigenvector of $H$ corresponding to the eigenvalue $\pm\i d_j$.
\end{proposition}
\begin{proof} 
The result immediately follows from the relation \eqref{eq:sympleigdef}.
\end{proof}

This proposition shows that the eigenvalues of a~pdH matrix $H$ are purely ima\-gi\-nary and that they can be determined by computing the symplectic eigenvalues of the corresponding spd matrix $M=J_{2n}^TH$.
\subsection{Characterization of the set of symplectically diagonalizing matrices}

As we mentioned before, the diagonalizing matrix in Williamson's diagonal form \eqref{eq:sympleigdecomp} is not unique. In this subsection, we aim to characterize the set of all symplectically diagonalizing matrices. 

First, note that if $M\in\SPD(2n)$ has only one symplectic eigenvalue of multiplicity $n$, then by Proposition~\ref{prop:Jain2020}{(ii)} such a~set is given by $S\OrS(2n)$,
where $S$ is any~symplectically diagonalizing matrix of $M$. For general case, we present two special classes of symplectically diagonalizing matrices.

\begin{proposition}\label{prop:ClassDiagonalizing}
Let $M\in\SPD(2n)$ and let $S\in\Spn$ symplectically diagonalize $M$. Then the following statements hold.
\begin{enumerate}
\item[i)] Let $R_{(j,\theta)}\in\R^{2n\times 2n}$ be the Givens rotation matrix of angle $\theta$ in the plane spanned by $e_j$ and $e_{n+j}$. Then $SR_{(j,\theta)}$ symplectically diagonalizes $M$ for any $j = 1,\ldots,n$ and $\theta \in [0,2\pi)$. 

\item[ii)] Let $Q=\diag(Q_1,\ldots,Q_q,Q_1,\ldots,Q_q)$, where $Q_j\in\mathbb{R}^{m_j\times m_j}$, $j = 1,\ldots,q$, are orthogonal, 
$m_1,\ldots,m_q$ are multiplicities of the symplectic eigenvalues and $m_1+\ldots+m_q = n$. Then $SQ$ symplectically diagonalizes $M$.
\end{enumerate} 
\end{proposition}

\begin{proof}
As the product of two symplectic matrices is again symplectic, we have to show that 
$R_{(j,\theta)}$ and $Q$ are symplectic, and that they congruently preserve $\diag(D,D)$, i.e., 
$R_{(j,\theta)}^T\diag(D,D) R_{(j,\theta)} = \diag(D,D) $ and 
$Q^T\diag(D,D) Q = \diag(D,D) $. This can be verified by direct calculations. 
\end{proof}

In the case $n=1$, it follows from \cite[Proposition~8.12]{deGo06} that the set of all symplectically diagonalizing matrices is $S \,\mathbb{SO}(2)$, where $\mathbb{SO}(2)$ is the  orthogonal group of rotations in $\mathbb{R}^2$. In other words, the first class of matrices in Proposition \ref{prop:ClassDiagonalizing} completely cha\-rac\-te\-rizes the set of all symplectically diagonalizing matrices when $n=1$.

For the general case $n>1$, it turns out that Proposition~\ref{prop:Jain2020} plays an important role in establishing the required result. Using the first statement in this proposition, we can show that the symplectic eigenvectors associated with distinct symplectic eigenvalues are linearly independent, see, e.g., \cite[Theorem~1.15]{deGo06}. 
Let $$
A^{(i)} = \begin{bmatrix}
A_1^{(i)}& A_2^{(i)}\\A_3^{(i)}&A_4^{(i)}
\end{bmatrix}
\in \mathbb{R}^{2k_{i}\times 2k_{i}},\qquad i = 1,\ldots,q,
$$
be matrices that have been decomposed into four square blocks. We will denote by
$$
\mathrm{dab}(A^{(1)},\dots,A^{(q)}) = \begin{bmatrix}
A_1& A_2\\A_3&A_4
\end{bmatrix}
$$
the $2(k_1+\cdots+k_q)\times 2(k_1+\cdots+k_q)$ matrix generated by diagonally assembling the blocks
$A^{(i)}_{\ell}$ such that $A_{\ell} = \diag(A^{(1)}_{\ell},\ldots,A^{(q)}_{\ell})$, $\ell= 1,\ldots,4$.  Hence the notation ``dab". 
If each matrix $A^{(i)}$ belongs to a set of matrices $\Phi^{(i)}$, then $\mathrm{dab}(\Phi^{(1)}\times \cdots \times \Phi^{(q)})$ denotes the set of all matrices $\mathrm{dab}(A^{(1)},\dots,A^{(q)})$ 
with $A^{(i)}\in \Phi^{(i)}$, $i = 1,\ldots,q$. It is straightforward to verify the following lemma.
\begin{lemma}\label{lem:tracingS}
For any set of integers $k_1,\ldots,k_q$, it holds that 
$$
\mathrm{dab}\bigl(\OrS(2k_1)\times \cdots \times \OrS(2k_q)\bigr)\subset \OrS\bigl(2(k_1+\cdots+k_q)\bigr).
$$
\end{lemma}

One can check that the matrices $R_{(j,\theta)}$ and $Q$ in Proposition~\ref{prop:ClassDiagonalizing} are elements of the set $\mathrm{dab}\bigl(\OrS(2k_1)\times \cdots \times \OrS(2k_q)\bigr)$ with appropriately chosen $k_1,\ldots,k_q$. Indeed, for any $j = 1,\ldots,n$, $R_{(j,\theta)} \in \mathrm{dab}\bigl(\OrS(2(j-1))\times \OrS(2)\times \OrS(2(n-j))\bigr)$. Similarly, the matrix $Q$ 
belongs to the set
 $\mathrm{dab}\bigl(\OrS(2m_1)\times \cdots \times \OrS(2m_q)\bigr)$.

We are now ready to state the main result in this subsection. Theorem \ref{prop:trackingS} below is an important improvement of the classical result \cite[Proposition~8.12]{deGo06} in the sense that it characterizes exactly the set of symplectically diagonalizing matrices of \mbox{$M\in \SPD(2n)$}. Moreover, its sufficiency part covers the matrix classes in Proposition~\ref{prop:ClassDiagonalizing} as special cases. Finally, it is also a~nontrivial generalization of Proposition~\ref{prop:Jain2020}{(ii)}.

\begin{theorem}\label{prop:trackingS}
Let $M\in \SPD(2n)$ have $q\leq n$ distinct symplectic eigenva\-lues $d_1,\ldots,d_q$ with multiplicities $m_1,\ldots,m_q$, respectively, and let $S\in \Spn$ be a~symplectically diagonalizing matrix of $M$. Then $T \in \Spn$ symplectically diagonalizes $M$ if and only if there exists 
$K \in \mathrm{dab}\bigl(\OrS(2m_1)\times \cdots \times \OrS(2m_q)\bigr)$ such that $T = SK$. 
\end{theorem}

\begin{proof}
First, we show the sufficiency. 
Lemma~\ref{lem:tracingS} implies that $K \in \OrS(2n)$. Then we obtain 
$$
T^TMT = K^TS^TMSK = K^T\diag(D,D)K = K^TK\diag(D,D) = \diag(D,D),
$$
where the third equality follows from the fact that 
$K\in \mathrm{dab}\bigl(\OrS(2m_1)\times \cdots \times \OrS(2m_q)\bigr)$. This means that $T$ symplectically diagonalizes $M$. 

Conversely, let 
$T$ symplectically diagonalize $M$. Let us pick any symplectic eigenvalue $d_i$ of multiplicity $m_i$, $i = 1,\ldots,q$, and let
${\mathcal{I}}_{m_i}=\left\{ j_i + 1,\ldots, j_i+m_i\right\}$
with $j_i=m_1+\cdots+m_{i-1}$.
Then the columns of $S_{\mathcal{I}_{m_i}},T_{\mathcal{I}_{m_i}}\in\mathbb{R}^{2n\times 2m_i}$ form the normalized symplectic eigenvector sets associated with $d_i$. Therefore, by Proposition~\ref{prop:Jain2020}{(ii)} there exists $K^{(i)}\in\OrS(2m_i)$ such that $T_{\mathcal{I}_{m_i}}=S_{\mathcal{I}_{m_i}}K^{(i)}$.
Ordering the columns of $T_{\mathcal{I}_{m_i}}$ and $S_{\mathcal{I}_{m_i}}$ for $i=1,\ldots,q$ as in $T$ and $S$, respectively, we obtain $T=SK$ with $K=\mathrm{dab}\bigl(K^{(1)},\ldots,K^{(q)}\bigr)\in
\mathrm{dab}\bigl(\OrS(2m_1)\times \cdots \times \OrS(2m_q)\bigr)$.
\end{proof}

\subsection{Computation of Williamson's diagonal form}
Here, we present an~algorithm based on \cite{Part13} for computing a symplectically diagonalizing matrix $S$ of $M$ in \eqref{eq:sympleigdecomp}. This procedure can also be viewed as 
a~constructive proof of Williamson's theorem. Since $M$ is spd, its real symmetric square root $M^{1/2}$ exists. It is easy to check that $\tilde{M}=M^{1/2}J_{2n}M^{1/2}$ is skew-symmetric and nonsingular.  This matrix can be transformed into the real Schur form 
\begin{equation}\label{eq:Schur}
Q^T\tilde{M}Q = \diag\left(\begin{bmatrix}
0&d_1\\-d_1&0
\end{bmatrix},\ldots, \begin{bmatrix}
0&d_n\\-d_n&0
\end{bmatrix}\right),
\end{equation}
where $Q$ is orthogonal, and $0<d_1\leq\ldots\leq d_n$, 
see \cite[Theorem~7.4.1]{GoluV13}. Further, let
\begin{equation}
P = [e_1,e_3,\ldots,e_{2n-1},e_2,e_4,\ldots,e_{2n}]
\label{eq:matrP}
\end{equation}
denote the perfect shuffle permutation matrix. Obviously, $QP$ is orthogonal and it holds 
$$
P^TQ^T\tilde{M}QP = \begin{bmatrix}
0&D\\-D& 0
\end{bmatrix},
$$
where $D = \diag(d_1,\ldots,d_n)$. Finally, we set
\begin{equation}\label{eq:DiagonalizingMat}
S = J_{2n}M^{1/2} QP\begin{bmatrix}
0&-D^{-1/2}\\D^{-1/2}& 0
\end{bmatrix}.
\end{equation}
It can be verified that $S$ is symplectic and $S^TMS = \diag(D,D)$.  
For ease of reference, we summarize these steps in Algorithm~\ref{alg:WilliamsonDiag}.

\begin{algorithm}
	\caption{Williamson's diagonal form}
	\label{alg:WilliamsonDiag}
	\begin{algorithmic}[1]
		\REQUIRE $M\in\SPD(2n)$.
		\ENSURE $S\in\Spn$, $D=\diag (d_1,\ldots,d_n)$ such that  $S^TMS = \diag(D,D)$.
		\STATE {Compute the symmetric square root $M^{1/2}$ of $M$.}
		\STATE Compute the real Schur form \eqref{eq:Schur} of $\tilde{M} = M^{1/2}J_{2n}M^{1/2}$. 
		\STATE Set $D=\mbox{diag}(d_1,\ldots,d_n)$ and compute the symplectic matrix $S$ as in \eqref{eq:DiagonalizingMat} with $P$ given in \eqref{eq:matrP}. 
	\end{algorithmic}
\end{algorithm}

Note that $M^{1/2}$ can be computed using the spectral decomposition of $M$, see \cite[Section~6.2]{High08}. For the computation of the real Schur form \eqref{eq:Schur}, we can employ the skew-symmetric QR algorithm \cite{WardGray78}. In this case, Algorithm~\ref{alg:WilliamsonDiag} requires about
$125n^3$ flops.

\subsection{Williamson's diagonal form for skew-Hamiltonian matrices}
To close this section, we present an~alternative algorithm for computing Williamson's diagonal form of spd matrices which are additionally assumed to be skew-Hamiltonian. This algorithm and Proposition~\ref{prop:Diagonalizing_Orthosympl} below will be of crucial importance and employed as a step, which is faster than Algorithm~\ref{alg:WilliamsonDiag} designed for general spd matrices, in our optimization method for computing the symplectic eigenvalues and eigenvectors
of general spd matrices presented in Section~\ref{sec:EigenCompviaRiemOpt}.

\begin{proposition}\label{prop:Diagonalizing_Orthosympl}
	Let $N\in\mathbb{R}^{2n\times 2n}$ be spd and skew-Hamiltonian.
	If $S$ symplectically diagonalizes $N$, then $S\in \OrS(2n)$.
\end{proposition}

\begin{proof}
It has been constructively shown in \cite{BennKM05}
that any skew-Hamiltonian matrix $N$ can be transformed into a real skew-Hamiltonian-Schur form 
\begin{equation}
K^TN K=\begin{bmatrix} \Omega_{11} & \Omega_{12}\\ 0& \Omega_{11}^T\end{bmatrix}, 
\label{eq:skewSchur}
\end{equation}
where $K\in \OrS(2n)$ and $\Omega_{11}\in\mathbb{R}^{n\times n}$ is quasi-triangular with diagonal blocks of order one and two corresponding, respectively, to real and complex standard eigenvalues of $N$. Since $N$ is spd, we obtain that $\Omega_{11}$ is diagonal and $\Omega_{12}=0$. Thus, $K$ symplectically diagonalizes $N$.

Let $S$ be any symplectically diagonalizing matrix of $N$ and let $K\in \OrS(2n)$ be the diagonalizing matrix as in \eqref{eq:skewSchur}. Then by \cite[Proposition~8.12]{deGo06}, 
	we have $K^{-1}S \in \OrS(2n)$.
	This implies that $S \in \OrS(2n)$.
\end{proof}

It immediately follows from Proposition~\ref{prop:Diagonalizing_Orthosympl} that the standard eigenvalues of
an~spd and skew-Hamiltonian matrix $N$ coincide with the symplectic eigenvalues. Moreover, we obtain
that the symplectically diagonalizing matrix
of $N$ constructed by Algorithm~\ref{alg:WilliamsonDiag} is orthosymplectic.

An alternative method for computing Williamson's diagonal form of $N$, based on the construction of the skew-Hamiltonian-Schur form \eqref{eq:skewSchur} as presented in \cite[Algorithm~10]{BennKM05}, is now summarized in Algorithm~\ref{alg:skewHamSchur}.
Note that this algorithm is strongly backward stable and costs about $23n^3$ flops.

\begin{algorithm}
	\caption{Williamson's diagonal form for spd and skew-Hamiltonian matrices}
	\label{alg:skewHamSchur}
	\begin{algorithmic}[1]
		\REQUIRE $N\in\mathbb{R}^{2n\times 2n}$ is spd and skew-Hamiltonian.
		\ENSURE $K\in\OrS(2n)$, $D=\diag (d_1,\ldots,d_n)$ such that  $K^TNK = \diag(D,D)$.
		\STATE Compute the symmetric Paige/Van Loan form $N=U \diag(\Omega_{1},\Omega_{1})U^T$ with \mbox{$U\in\OrS(2n)$}
		and tridiagonal $\Omega_{1}\in \SPD(n)$ as described in \cite{VanL84}.
		\STATE Compute the symmetric Schur form  $\Omega_{1}= Q_1DQ_1^T$, where $Q_1$ is orthogonal and
$D=\diag(d_1,\ldots, d_n)$.  
		\STATE Compute the orthosymplectic matrix $K=U \diag(Q_1,Q_1)$. 
	\end{algorithmic}
\end{algorithm}

\section{Symplectic trace minimization problem}\label{sec:SymplTraceMin} 
In this section, we establish the connection between the symplectic EVP and the symplectic trace minimization problem. The following result is one of the main sources that inspire our work. 

\begin{theorem}\label{theo:SymplMin}\textup{{\cite{Hiro06,BhatJ15}}}
Let a~matrix $M\in\SPD(2n)$ have symplectic eigenvalues \mbox{$d_1\leq\ldots\leq d_n$}. Then for any integer $1\leq k \leq n$, it holds
\begin{equation}\label{eq:SymplOptimProb}
2\sum_{j=1}^{k}d_j = \min\limits_{X\in\R^{2n\times 2k}} f(X):= \tr(X^TMX)\quad \mbox{s.t.}\quad h(X):=X^TJ_{2n}X -J_{2k}=0.
\end{equation}
\end{theorem}

Due to the constraint condition, the problem \eqref{eq:SymplOptimProb} can be viewed as 
the minimization problem restricted to the symplectic Stiefel manifold $\Spkn$. The following lemma
establishes the homogeneity of the cost function $f$ on $\OrS(2k)$.

\begin{lemma}\label{lem:homog}
Let $M\in\SPD(2n)$. For $X\in \Spkn$ and $K \in \OrS(2k)$, the cost function $f$ in 
\eqref{eq:SymplOptimProb} satisfies $f(XK) = f(X)$.
\end{lemma}


%
\begin{proof}
For $X\in \Spkn$ and $K \in \OrS(2k)$, we obtain that $XK \in \Spkn$ and 
$$f(XK) = \mbox{tr}(K^TX^TMXK) = \mbox{tr}(K^{-1}X^TMXK) = \mbox{tr}(X^TMX) = f(X).$$
Here, we used the fact that similar matrices have the same trace.
\end{proof}

\subsection{Critical points}
First, we investigate the critical points of the optimization problem \eqref{eq:SymplOptimProb}. 
For this purpose, we will invoke the associated Lagrangian function
\begin{equation*}
\mathcal{L}(X,L) = \tr(X^TMX) - \tr(L(X^TJ_{2n}X-J_{2k})),
\end{equation*}
where $L\in \mathbb{R}^{2k\times 2k}$ is the Lagrangian multiplier. Since the constraint function 
$h$ maps $\mathbb{R}^{2n\times 2k}$ into $\mathcal{S}_{\rm skew}(2k)$, 
the Lagrangian multiplier $L$ can also be taken skew-symmetric. 
The gradient of $\mathcal{L}$ with respect to the first argument at $(X,L)$ takes the form
\begin{equation}\label{eq:gradLagrange}
\nabla_X\mathcal{L}(X,L) =  2MX - 2J_{2n}XL.
\end{equation}
Furthermore, the action of the Hessian of $\mathcal{L}$ with respect to the first argument on $(W,W) \in \Rbnk \times \Rbnk$ reads
\begin{equation}\label{eq:HessLagrange}
\nabla^2_{XX}\mathcal{L}(X,L)[W,W] = 2\,\tr\bigl(W^T(MW -J_{2n}W L)\bigr).
\end{equation}

Next, let us recall the first- and the second-order necessary optimality conditions \cite{NoceW06} for the constrained optimization problem \eqref{eq:SymplOptimProb}. A point $X_*\in \Rbnk$ is called 
a~\emph{critical point} of the problem \eqref{eq:SymplOptimProb} if $h(X_*)=0$ and there exists 
a~Lagrangian multiplier $L_* \in \mathcal{S}_{\rm skew}(2k)$ such that $\nabla_X\mathcal{L}(X_*,L_*) = 0$. These conditions are
known as the Karush-Kuhn-Tucker conditions. The first condition implies that $X_*\in\Spkn$. Using \eqref{eq:gradLagrange}, the stationarity condition
can equivalently be written as
\begin{equation}\label{eq:KKTcond}
MX_* = J_{2n}X_*L_*.
\end{equation}
 Comparing \eqref{eq:SymplEigvectorSet_2} with \eqref{eq:KKTcond}, we obtain that any normalized symplectic eigenvector set $X$ of $M$ is a critical point with the  Lagrangian multiplier 
	$$L_*=\left[\begin{matrix}
		0 & -\Lambda\\
		\Lambda & \enskip 0
\end{matrix}\right].$$
In this case, multiplying \eqref{eq:KKTcond} with $X_*^T$ on the left and taking the trace of the resulting equality lead to 
\begin{equation}\label{eq:valueateigenvectorset}
f(X_*) = 2\tr(\Lambda)=2(\lambda_1+\cdots+\lambda_k).
\end{equation} 
The critical point $X_*\in \Rbnk$ with the associated Lagrangian multiplier $L_*$ is said to satisfy the second-order necessary optimality condition if 
\begin{equation*}
\nabla^2_{XX}\mathcal{L}(X_*,L_*)[W,W] = 2\,\tr\bigl(W^T(MW -J_{2n}W L_*)\bigr)\geq 0
\end{equation*}
for all \mbox{$W\! \in \mbox{null}\bigl(\mathrm{D} h(X_*)\bigr)\!:=\!\{Y\! \in\! \mathbb{R}^{2n\times 2k}: 
 \mathrm{D} h(X_*)[Y]=Y^T\!J_{2n}X_* \!+\!X_*^TJ_{2n}Y\! = 0\}$}.

Based on Proposition~\ref{prop:Diagonalizing_Orthosympl}, we can characterize the critical points
of the optimization problem \eqref{eq:SymplOptimProb} as follows.

\begin{theorem}\label{prop:Critical_Orthosympl}
	Let $M\in\SPD(2n)$.
\begin{itemize}
\item[i)] If $X_* \in \Spkn$ is a critical point of \eqref{eq:SymplOptimProb}, then for any $K \in \OrS(2k)$, the matrix $X_*K$ is also a critical point of \eqref{eq:SymplOptimProb}.
\item[ii)] A~matrix $X_*\in \Spkn$ is a critical point of \eqref{eq:SymplOptimProb} if and only if there exists $K \in \OrS(2k)$ such that the columns of $X_*K$ form a~normalized symplectic eigenvector set of $M$. 
\end{itemize}
\end{theorem}

\begin{proof}
i) If $X_*\in \Spkn$ is a~critical point of \eqref{eq:SymplOptimProb} with the associated Lagrangian multiplier $L_*$, then \eqref{eq:KKTcond} is fulfilled. Therefore, for any $K\in \OrS(2k)$, we obtain that $X_*K\in\Spkn$ and 
$
MX_*K = J_{2n}X_*L_*K = J_{2n}X_*KK^TL_*K$.
This means that $X_*K$ is also a critical point of \eqref{eq:SymplOptimProb} with the Lagrangian multiplier $K^TL_*K$.

ii) Assume that the columns of $Z_*=X_*K$ with $K\in\OrS(2k)$ form a normalized symplectic eigenvector set of $M$. Then $Z_*\in \Spkn$ is a critical point of \eqref{eq:SymplOptimProb}, and, hence, by i), 
$X_*=Z_*K^{-1}\in \Spkn$ is also a~critical point of \eqref{eq:SymplOptimProb}.

Conversely, let $X_*\in \Spkn$ be a~critical point  of \eqref{eq:SymplOptimProb}. 
Then $X_*$ satisfies  \eqref{eq:KKTcond} which immediately implies that
\begin{equation}\label{eq:proofCritical_Orthosympl_1}
X_*^TMX_* = J_{2k}L_*
\end{equation}
with a~skew-symmetric matrix $L_*$. We now show that $X_*^TMX_*^{}$ is spd and skew-Hamiltonian. Since
$M$ is spd and $X_*$ has full column rank, we obtain that $X_*^TMX_*^{}$ is spd. Furthermore, using 
 \eqref{eq:proofCritical_Orthosympl_1}, we get
$$
(J_{2k}X_*^TMX_*^{})^T= (J_{2k}J_{2k}L_*)^T=L_*=-J_{2k}J_{2k}L_*=-J_{2k}X_*^TMX_*^{}
$$ 
implying that $X_*^TMX_*^{}$ is skew-Hamiltonian.
Then by Propostion~\ref{prop:Diagonalizing_Orthosympl}, there exists $K\in\OrS(2k)$ 
such that 
\begin{equation}\label{eq:proofCritical_Orthosympl_2}
K^T(X_*^TMX_*)K = \begin{bmatrix} \Lambda & 0 \\ 0 & \Lambda\end{bmatrix}
\end{equation}
with $\Lambda = \mbox{diag}(\lambda_1,\ldots,\lambda_k)$. Using \eqref{eq:KKTcond}, \eqref{eq:proofCritical_Orthosympl_1}, \eqref{eq:proofCritical_Orthosympl_2} and $J_{2k}^TK=KJ_{2k}^T$, we deduce
\begin{equation*}
\arraycolsep=2pt
\begin{array}{rcl}
MX_*K & = & J_{2n}X_*L_*K = J_{2n}X_*J_{2k}^TKK^TJ_{2k}^{}L_*K\\
& = & J_{2n}X_*KJ_{2k}^TK^TX_*^TMX_*^{}K= J_{2n}X_*K\begin{bmatrix} 0&-\Lambda \\  \Lambda&\enskip 0\end{bmatrix}.
\end{array}
\end{equation*}
Thus, the columns of $X_*K$ form a normalized symplectic eigenvector set of $M$.
%
\end{proof}

Theorem~\ref{prop:Critical_Orthosympl} allows us
to characterize the set of all critical points of the problem \eqref{eq:SymplOptimProb}, and particularly the set of all minimizers as we will see in the next subsection.
%


\begin{corollary}\label{corol:CriticalClasses}
	The set of all critical points of the minimization problem \eqref{eq:SymplOptimProb} is the union of all $X\OrS(2k)$, where the columns of $X\in\Spkn$ form any possible normalized symplectic eigenvector set of $M$. 
\end{corollary}

\begin{remark} 
We can extend Theorem~\textup{\ref{prop:trackingS}} to the case $S\in\Spkn$ with $k<n$ by the same proof. Now, the picture is clear. We have three different tools to track different objects: $\Spk$ for tracking the symplectic matrices that span the same subspace {\rm(}Proposition~\textup{\ref{prop:SymplGroup}{(iii)}}{\rm)}, the ``{\normalfont dab}" set for the symplectically diagonalizing matrices of $M$ {\rm (}Theorem~\textup{\ref{prop:trackingS}}{\rm )},  and 
$\OrS(2k)$ for the set of feasible points at which the value of the cost function $f$ in \eqref{eq:SymplOptimProb} is the same {\rm (}Lemma~\textup{\ref{lem:homog}}{\rm )} and for the set of all critical points of \eqref{eq:SymplOptimProb} {\rm (}Theorem~\textup{\ref{prop:Critical_Orthosympl}}{\rm )}.
\end{remark}

\subsection{Local and global minimizers}
We now investigate the local and global minimizers of the optimization problem \eqref{eq:SymplOptimProb}.

\begin{theorem}\label{prop:minimizers_Orthosympl}
	Let $M\in\SPD(2n)$.
\begin{itemize}
\item[i)] If $X_* \in \Spkn$ is a global minimizer of \eqref{eq:SymplOptimProb}, then for any $K \in \OrS(2k)$, the matrix $X_*K$ is also a global minimizer of \eqref{eq:SymplOptimProb}.
\item[ii)] A~matrix $X_*\in \Spkn$ is a global minimizer of \eqref{eq:SymplOptimProb} if and only if there \mbox{exists} $K \in \OrS(2k)$ such that the columns of $X_*K$ form
a~normalized symplectic eigenvector set of $M$ associated with the symplectic eigenvalues 
$d_1,\ldots,d_k$. 
\end{itemize}
\end{theorem}

\begin{proof}
i) Let $X_*\in \Spkn$ is a global minimizer of \eqref{eq:SymplOptimProb} and let $K \in \OrS(2k)$. Then 
$X_*K\in \Spkn$. Furthermore, by Lemma~\ref{lem:homog} we obtain  
$f(X_*K)=f(X_*)$, and, hence, $X_*K$ is a global minimizer of \eqref{eq:SymplOptimProb}.

ii) In view of Lemma~\ref{lem:homog} and \eqref{eq:valueateigenvectorset}, the sufficiency immediately follows from $f(X_*)=f(X_*K)=2(d_1 + \cdots + d_k)$ for any $K\in \OrS(2k)$. Conversely, if $X_*$ is a~mi\-ni\-mi\-zer, it must be a critical point. 
Due to Theorem~\ref{prop:Critical_Orthosympl}, there exists $K \in \OrS(2k)$ such that $X_*K$ is a normalized symplectic eigenvector set corresponding to a~set of symplectic eigenvalues, say $d_{i_1},\ldots,d_{i_k}$. 
Taking again Lemma~\ref{lem:homog} and \eqref{eq:valueateigenvectorset} into account, we deduce from this fact that 
$$2(d_1 + \cdots + d_k) = f(X_*) = f(X_*K) =  2(d_{i_1}+\cdots+d_{i_k}).$$ 
Because all $d_{i_j}$,~$j=1,\ldots,k$, are taken from the set of positive numbers, where $d_i$, $i=1,\ldots,k$, are the $k$ smallest ones, we can conclude, after a reordering if necessary, that $d_{i_j} = d_j$ for $j=1,\ldots,k$.
\end{proof}

In Appendix~\ref{app:Anotherway}, we present an~alternative proof of the necessity in Theorem~\ref{prop:minimizers_Orthosympl}(ii) which does not rely on Theorem~\ref{prop:Critical_Orthosympl}.

Similarly to Corollary~\ref{corol:CriticalClasses}, we can now characterize the set of global minimizers of the problem
\eqref{eq:SymplOptimProb}.
 
\begin{corollary}\label{corol:MimimizersClasses}
The set of all global minimizers of \eqref{eq:SymplOptimProb} is the union of all $X_{1:k}\OrS(2k)$, where the columns of $X_{1:k}\in\Spkn$ form a~normalized symplectic eigenvector set of $M$ associated with the symplectic eigenvalues $d_1,\ldots,d_k$. 
\end{corollary}

\begin{remark}\label{rm:complete}
If $d_k<d_{k+1}$, Corollary~\textup{\ref{corol:MimimizersClasses}} can be considered as a symplectic version of the corresponding result for the standard EVP, see, e.g., \textup{\cite[Theorem~2.1]{SameT00}}. In this case, $X_{1:k}$ can be constructed by taking the $1$-st, $\ldots$, $k$-th, $(n+1)$-st, $\ldots$, $(n+k)$-th columns of any symplectically diagonalizing matrix $S$ of $M$. 
Otherwise, let $j$ be the largest index such that $d_j<d_k$. Then, the last $k- j$ columns in the first and second halves of $X_{1:k}$ can be any of those whose column indices are ranging from  $j+1$ to $j+ m_k$ and their counterparts in the second half of $S$, where $m_k$  
denotes the multiplicity of $d_k$. In all related statements in the rest of this paper, by $X_{1:k}$, we include all such cases.
\end{remark}

Next, we collect some consequences from Theorem~\ref{prop:minimizers_Orthosympl} for the case $k=n$.

\begin{corollary}\label{corol:Mimimizersk=n}
Let $M\in\SPD(2n)$. 
\begin{itemize}
\item[i)] Any critical 
point of the minimization problem \eqref{eq:SymplOptimProb} with $k=n$
is a global minimizer.
\item[ii)] The set of all global minimizers of \eqref{eq:SymplOptimProb} with $k=n$
	is $S\, \OrS(2n)$, where $S\in\Spn$ is a~symplectically diagonalizing matrix of $M$.
\end{itemize}
\end{corollary}


We now consider the non-existence of non-global local minimizers. In view of Corollary~\ref{corol:Mimimizersk=n}, we restrict ourselves to the case $k<n$. A similar result for 
the gene\-ra\-li\-zed EVP can be found in \cite{KovaV95,LianLB13}. First, we state
an~important technical lemma.
\begin{lemma}\label{lem:TangentNull}
Let $M\in\SPD(2n)$ and let the columns of $X_{1:k}$ and $X_{n-k+1:n}$ form any normalized symplectic eigenvector sets associated, respectively, with the $k$ smallest and $k$ largest symplectic eigenvalues of $M$. Then for
any critical point $X_0$ of the optimization problem \eqref{eq:SymplOptimProb},
there exist
a~global minimizer $X_*\in X_{1:k}\OrS(2k)$ and an~$X^*\in X_{n-k+1:n}\OrS(2k)$ such that $X_*, X^* \in \mathrm{null}\bigl(\mathrm{D} h(X_0)\bigr)$.
\end{lemma}
\begin{proof} See Appendix \ref{app:proofLemTangentNull}.
\end{proof}

\begin{proposition}\label{prop:localnonglobal}
Every local minimizer of the optimization problem \eqref{eq:SymplOptimProb} is a~global one.
\end{proposition} 

\begin{proof}
Assume that there is a non-global local minimizer $X_0$ of the problem \eqref{eq:SymplOptimProb}. 
Since $X_0$ is a~critical point, there is an associated Lagrangian multiplier $L_0$. Moreover, by Corollary~\ref{corol:CriticalClasses}, 
$X_0$ can be represented as \mbox{$X_0 = X_{\mathcal{I}_k}K_0$},  where 
$K_0 \in \OrS(2k)$, and the columns of $X_{\mathcal{I}_k}$ form a~normalized symplectic eigenvector set associated with a~set of the symplectic eigenvalues $\{d_{i_j}, i_j\in\mathcal{I}_k\}$ in which at least one of them is greater than $d_k$. 
By Lemma~\ref{lem:TangentNull}, there exists a global minimizer $X_* \in \mathrm{null}\bigl(\mathrm{D} h(X_0)\bigr)$. On the account of \eqref{eq:HessLagrange}, we get then 
\begin{align*}
&\nabla^2_{XX}\mathcal{L}(X_0,L_0)[X_*,X_*] = 2\,\tr(X_*^TMX_* - X_*^TJ_{2n}X_*L_0) 
= 2\,\tr(X_*^TMX_* -J_{2k}L_0)\\
&\qquad= 2\,\tr(X_*^TMX_* -X_0^TMX_0) = 4\,\sum_{i=1}^k d_i -4\,\sum_{j=1}^k d_{i_j} <0,
\end{align*}
which contradicts to the second-order necessary optimality condition for $X_0$. This completes the proof.
\end{proof}

Saddle points of the cost function $f$ in the problem \eqref{eq:SymplOptimProb} can be disclosed in the following.

\begin{proposition}\label{prop:SaddlePoint}
Any normalized symplectic eigenvector set $X$ of a~matrix $M\in\SPD(2n)$ associated with a symplectic eigenvalue set $\{d_{i_j}, i_j\in \calI_k\}$, in which there is at least one $d_{i_j}$ such that 
$d_k < d_{i_j} < d_{n-k+1}$, is a saddle point of \eqref{eq:SymplOptimProb}.
\end{proposition}
\begin{proof}
Obviously, $X$ is a critical point. Then it follows from the proof of Proposition~\ref{prop:localnonglobal} that $X$ is not a minimizer. Taking into account the existence of $X^*$ in Lemma~\ref{lem:TangentNull} and following the same proof of Proposition~\ref{prop:localnonglobal}, we can show that $X$ is not a maximizer of the cost function $f$ in \eqref{eq:SymplOptimProb} either. Hence, $X$ is a saddle point.
\end{proof}

\begin{remark}\label{rem:nomaxmizer}
Unfortunately, we were unable to prove that each element in the matrix set $X_{n-k+1:n}\OrS(2k)$ is a local maximizer. Nevertheless, we can show that $f$ in \eqref{eq:SymplOptimProb} has no global maximizer. Indeed, let us consider a~symplectic matrix
$$
X_a = \begin{bmatrix}
aI_{n,k} & 0\\0 & 1/aI_{n,k}
\end{bmatrix}, \quad a\not= 0,
$$
where $I_{n,k}$ denotes a $n\times k$ submatrix of $I_n$.
For any symplectically diagonalizing matrix $S$ of $M$, $SX_a \in \Spkn$. We then get that
$$
f(SX_a) = \tr(X_a^TS^TMSX_a) = 2(a^2+\frac{1}{a^2})\tr(I_{n,k}^TDI_{n,k}^{})
$$
which tends to infinity when $a\to 0$.
\end{remark}

We close this section by considering some consequences for the case $k=1$. 

\begin{corollary}\label{coroll:critical_minimizer}
	Let $M\in\SPD(2n)$ be in Williamson's diagonal form \eqref{eq:sympleigdecomp}.
	\begin{itemize}
		\item[i)] 
		The two columns of $X \in \mathrm{Sp}(2,2n)$ form a normalized symplectic eigenvector pair of $M$ if and only if $X$ is a critical point of the minimization problem \eqref{eq:SymplOptimProb} with $k=1$.
		\item[ii)] The two columns of $X_1\in \mathrm{Sp}(2,2n)$ form a~normalized symplectic eigenvector pair of $M$ associated with the smallest eigenvalue $d_1$ if and only if $X_1$ is a~global minimizer of \eqref{eq:SymplOptimProb} with $k=1$.
		\item[iii)] For any $j=2,\ldots,n-1$ such that $d_1 < d_j < d_n$, a~normalized symplectic eigenvector pair $X_j\in \mathrm{Sp}(2,2n)$ of $M$ associated with $d_j$ is a~saddle point of \eqref{eq:SymplOptimProb} with $k=1$.
	\end{itemize}
\end{corollary}

Corollary \ref{coroll:critical_minimizer} can be considered as a symplectic version of the corresponding results on the trace minimization problem for standard eigenvalues. Especially, part~(i) is similar to \cite[Proposition~4.6.1]{AbsiMS04}; part~(ii) is similar to \cite[Proposition~4.6.2(i)]{AbsiMS04} with the note that $X_{1}$ is not unique; part~(iii) is the same as \cite[Proposition~4.6.2(iii)]{AbsiMS04}. 

\section{Eigenvalue computation via Riemannian optimization}\label{sec:EigenCompviaRiemOpt}
In this section, we present a numerical method for solving the optimization problem \eqref{eq:SymplOptimProb}. It is principally a constrained optimization problem for which some existing methods can be used, see, e.g., \cite{NoceW06}. Nevertheless, maintaining the constraint is challenging. Recently, it has been shown in \cite{GSAS20} that the feasible set $\Spkn$ constitutes a Riemannian manifold. Moreover, two efficient methods were proposed there for optimization on this manifold. 
In this section, we briefly review the necessary ingredients for a~Riemannian optimization algorithm for solving \eqref{eq:SymplOptimProb} and discuss the computation of the smallest symplectic eigenvalues and
	the corresponding symplectic eigenvectors by using the presented optimization algorithm.

\subsection{Riemannian optimization on the symplectic Stiefel manifold}\label{Subsec:RiemOptStiefelMan}
Given $X\in\Spkn$, 
the tangent space of $\Spkn$ at $X$, denoted by $\TX $, can be represented as
	$\TX = \{AJ_{2n}X\; : \; A\in \mathcal{S}_{\rm sym}(2n) \}$, see~\cite[Proposition~3.3]{GSAS20} 
	for detail.
	In view of~\cite[Proposition~4.1]{GSAS20}, a Riemannian metric for $\Spkn$, called the canonical-like metric, is defined as
	\begin{align*}
	g_{\rho}(Z_1,Z_2) :=  \tr\left(  Z_1^T  \left( \frac{1}{\rho} J_{2n}XX^T J_{2n}^T -(J_{2n}XJ_{2k}X^T J_{2n}^T -J_{2n})^2 \right)  Z_2     \right),
	\end{align*}
	where $Z_1,Z_2\in\TX$ and $\rho>0$. Consequently, the associated Riemannian gradient 
	of the cost function $f$ in \eqref{eq:SymplOptimProb} has the following expression.
	
	\begin{proposition}\label{proposition:rgrad}
		Given $M\in\SPD(2n)$, the Riemannian gradient of the~function $f:\Spkn\to\mathbb{R}: X\mapsto\tr(X^T M X)$ associated with the metric $g_\rho$ is given by
		$\rgrad{X} = A_X J_{2n}X$ 
				with the matrices $A_X = 4\,\sym\left({H_X MX} (XJ_{2k})^T\right)$ and ${H_X}=I+\frac{\rho}{2}XX^T - J_{2n}X(X^T X)^{-1}X^T J_{2n}^T$.
	\end{proposition}
	\begin{proof}
		The result directly follows from $\nabla f(X)=2MX$ and~\cite[Proposition~4.5]{GSAS20}.
	\end{proof}
	
	In~\cite{GSAS20}, two searching strategies relying on quasi-geodesics and symplectic Cayley transform 
	were proposed for the optimization on $\Spkn$. It has also been shown there that the Cayley-based method performs better than that based on quasi-geodesics. Therefore, we choose the Cayley retraction as the update formula. Specifically, the searching curve along $-\rgrad{X}\in\TX$ is defined as
	\begin{equation}\label{eq:cay-Lie}
	\mathcal{R}_X(-t\,\rgrad{X}) := \left({I + \frac{t}{2}A_XJ_{2n}}\right)^{-1}\left({I -\frac{t}{2}A_XJ_{2n}}\right)X,
	\end{equation}
	where $A_X$ is as in Proposition~\ref{proposition:rgrad}. Note that since the number $k$ of required symplectic eigenvalues is usually small, the update~\eqref{eq:cay-Lie} can be further assembled in an~efficient way suggested in~\cite[Proposition~5.4]{GSAS20}. 
	
	In Algorithm~\ref{alg:non-monotone gradient}, we present the Riemannian gradient method with non-monotone line search for solving~\eqref{eq:SymplOptimProb}. Practically, we can stop the iteration when the gradient of the cost function is smaller than a~given tolerance $\epsilon$. It has been proven in \cite[Theorem~5.6]{GSAS20} that with standard assumptions, Algorithm~\ref{alg:non-monotone gradient} generates an infinite sequence of which any accumulation point is a critical point of \eqref{eq:SymplOptimProb}. 
	\begin{algorithm}
		\caption{Riemannian gradient method for solving the problem~\eqref{eq:SymplOptimProb}}
		\label{alg:non-monotone gradient}
		\begin{algorithmic}[1]
			\REQUIRE $M \in \SPD(2n)$; $X^{(0)}\in\Spkn$; 
			$\rho>0$, $\beta, \delta\in(0,1)$, $\alpha \in [0,1]$, 
			$q_0=1$, $c_0 = f(X^{(0)})$, $\gamma_0>0$, $0<\gamma_{\min}<\gamma_{\max}=1$; 
			\ENSURE Sequence of iterates $\{X^{(m)}\}$. 
			\FOR{$m=0,1,2,\dots$}
			\STATE Set $Z^{(m)} = -\rgrad{X^{(m)}}$.
			\IF{$m>0$} 
			\STATE{ 
			$
			\gamma_m=\left\{\begin{array}{ll}
				\frac{\jkh{W^{(m-1)}, W^{(m-1)}}}{\abs{\jkh{W^{(m-1)}, Y^{(m-1)}}}} &\text{for odd } m, \\[3mm]
				\frac{\abs{\jkh{W^{(m-1)}, Y^{(m-1)}}}}{{\jkh{Y^{(m-1)}, Y^{(m-1)}}}} &\text{for even } m,
				\end{array}\right.
			$\\
			where $W^{(m-1)} = X^{(m)} - X^{(m-1)}$ and $Y^{(m-1)} =Z^{(m)} -  Z^{(m-1)}$.}
			\ENDIF
			\STATE Calculate the trial step size $\gamma_k=\max\bigl(\gamma_{\min},\min(\gamma_k,\gamma_{\max})\bigr)$.			
			\STATE Find the smallest integer $\ell$  
			such that the non-monotone condition 
			$$f\dkh{{\cal R}_{X^{(m)}}(t_m Z^{(m)})} \le c_m + \beta\, t_m\, g_\rho\dkh{\rgrad{X^{(m)}},Z^{(m)}}$$
			holds, where $t_m=\gamma_m \delta^{\ell}$. 
			\STATE Set $X^{(m+1)} = {\cal R}_{X^{(m)}}(t_m Z^{(m)})$.
			\STATE Update $q_m = \alpha q_{m-1} + 1$ and
			$c_m = \frac{\alpha q_{m-1}}{q_m} c_{m-1}  + \frac{1}{q_m} f(X^{(m)})$.
			\ENDFOR
		\end{algorithmic}
	\end{algorithm}

\subsection{Computing the symplectic eigenvalues and eigenvectors}\label{subsec:EigenComp}
First, we consider the computation of the smallest symplectic eigenvalue $d_1$ of $M$.  This case was briefly addressed in \cite{GSAS20} as an example. We review it here and discuss the computation of the corresponding normalized symplectic eigenvector pair.
Let $X_*\in\mbox{Sp}(2,2n)$ be a~minimizer computed by Algorithm~\ref{alg:non-monotone gradient}. Then we have \mbox{$d_1= f(X_*)/2$} and by Corollary~\ref{coroll:critical_minimizer}(ii) the columns of $X_*$ provide the sought normalized symplectic eigenvector pair. 

We now consider the general case $1\leq k\leq n$. Assume that $X_*$ is a minimizer of \eqref{eq:SymplOptimProb}. According to Theorem~\ref{prop:minimizers_Orthosympl}(ii), there exists $K\in \OrS(2k)$ such that 
the columns of $X_*K$ form a~normalized symplectic eigenvector set of $M$ associated with the symplectic eigenvalues $d_1,\ldots,d_k$. 
The sought matrix $K$ can be computed by symplectically diagonalizing a $2k\times 2k$ matrix $X_*^TMX_*$. As $X_*^TMX_*$ is spd and skew-Hamiltonian, we can resort to Algorithm~\ref{alg:skewHamSchur}  for the sake of efficiency.
We summarize the computation of the $k$ smallest symplectic eigenvalues of $M$ and the corresponding eigenvector set in Algorithm~\ref{alg:SymplEigProb}.

\begin{algorithm}
	\caption{Symplectic EVP via Riemannian optimization}
	\label{alg:SymplEigProb}
	\begin{algorithmic}[1]
		\REQUIRE $M\in\SPD(2n)$, $1 \leq k \leq n$.
		\ENSURE $k$ smallest symplectic eigenvalues $d_1,\ldots,d_k$ and the corresponding norma\-lized symplectic eigenvector set $X_{1:k}\in\Spkn$.
		\STATE Solve the optimization problem~\eqref{eq:SymplOptimProb} for $X_*\!\!\in\!\Spkn$ by using Algorithm~\ref{alg:non-monotone gradient}.
		\STATE Compute Williamson's diagonal form $X_*^TMX_*^{}=K\diag(D_{1:k},D_{1:k})K^T$ with 
		\mbox{$K\in\OrS(2k)$} and $D_{1:k}=\diag(d_1,\ldots,d_k)$ by using Algorithm~\ref{alg:skewHamSchur}.
		\STATE Compute $X_{1:k} = X_*K$.
	\end{algorithmic}
\end{algorithm}

 Algorithm~\ref{alg:SymplEigProb} is comparable with typical methods for large standard EVPs in the sense that we first simplify and/or reduce the size of the problem and 
 then solve the small and/or simpler (symplectic) EVP. This approach 
 may be not efficient if all symplectic eigenvalues are required. In that case,  Algorithm~\ref{alg:WilliamsonDiag}, for instance, could be used. 

\begin{remark}\label{rem:maxf}
	Unlike the standard eigenvalue trace minimization problem on the Stiefel manifold, as shown in Remark~\textup{\ref{rem:nomaxmizer}}, the cost function $f$ in \eqref{eq:SymplOptimProb} is unbounded from above. This comes from the fact that the Stiefel manifold is bounded while the symplectic Stiefel manifold is not. Therefore, we cannot find largest symplectic eigenvalues in a~similar manner, i.e., by maximizing the cost function.
		Despite this fact, the largest symplectic eigenvalues of an spd matrix $M$ can be computed by applying Algorithm~\textup{\ref{alg:SymplEigProb}} to the inverse of $M$. As in the standard case, this follows from the fact that the largest eigenvalues of $M$ are the reciprocals of the corresponding smallest ones of its inverse \textup{\cite[Theorem~8.14]{deGo06}}. This task can be done as long as the linear equation $Mx = y$ can be solved efficiently. 
\end{remark}

\subsection{Computing the eigenvalues of positive-definite Hamiltonian matrices}
\label{sec:EigSPDHamilton}
As an~application of Algorithm~\ref{alg:SymplEigProb}, we consider the computation of standard eigenvalues and their corresponding eigenvectors of pdH matrices. Due to nume\-rous applications, the EVPs for general Hamiltonian matrices have attracted a~lot of attention
and many different algorithms were developed for such problems, e.g., \cite{VanL84,BennF97,Watk04,BennKM05,BennFS11}, just to name a~few. It is noteworthy that some of these methods rely on the Hamiltonian-Schur form. Unfortunately, this form does not always exist, e.g., for real Hamiltonian matrices having purely imaginary eigenvalues, which is exactly the case for pdH matrices, see Proposition~\ref{prop:SEVandEV}. In \cite{Amod03, Amod06}, a~symplectic Lanczos method was developed for computing a few extreme eigenvalues of a~pdH matrix $H$, which exploits the symmetry and positive definiteness of its gene\-ra\-tor $M=J_{2n}^TH$.

 
Here, we present a different numerical approach for computing the eigenvalues of pdH matrices which relies on Riemannian optimization. To the best of our know\-ledge, this is the first geometric method for the special Hamiltonian EVP. Based on Proposition~\ref{prop:SEVandEV}, we propose  
to compute the smallest (in modulus) eigenvalues of a~pdH matrix $H$ by applying Algorithm~\ref{alg:SymplEigProb} to the spd matrix $M=J_{2n}^TH$.

\section{Numerical examples}\label{Sec:NumerExam}
In this section, we present some results of numerical experiments 
demonstrating the proposed Riemannian trace minimization method, henceforth called Riemannian.
The parameters in Algorithm~\ref{alg:non-monotone gradient} are set to default values as given in \cite{GSAS20}.
Although accumulation points of the iterates generated by this algorithm 
can be proven to be critical points of the cost function in \eqref{eq:SymplOptimProb} only \cite{GSAS20}, we never experience stagnation at a~saddle point.  
This fact was observed in various works and arguably explained, see \cite{LeePPSJR19} and references therein. For reference and comparison, we also report the corresponding results for
the restarted symplectic Lanczos algorithm \cite{Amod06} (symplLanczos) and 
the MATLAB function \texttt{eigs} applied to the associated Hamiltonian matrix.
All computations were done on a workstation with two Intel(R) Xeon(R) Processors Silver 4110 (at 2.10GHz$\times 8$, 12M Cache) and 384GB of RAM running 
MATLAB R2018a under Ubuntu 18.10. The code that produced the results is available from
\href{https://github.com/opt-gaobin/speig}{https://github.com/opt-gaobin/speig}.

The accuracy of computed symplectic eigenvalues and eigenvector sets of $M$ are measured
by using the normalized residual
	$$
	\frac{\left\|M \tilde{X}_{1:k}-J_{2n}\tilde{X}_{1:k}\begin{bmatrix}0 & -\tilde{D}_{1:k}\\ \tilde{D}_{1:k} & 0\end{bmatrix}\right\|_F}{\|M\tilde{X}_{1:k}\|_F},
	$$
	where $\tilde{X}_{1:k}$ is the computed symplectic eigenvector set associated with the symplectic eigenvalues on the diagonal of $\tilde{D}_{1:k}=\diag(\tilde{d}_1,\ldots,\tilde{d}_k)$. Here, $\|\cdot\|_F$ denotes the Frobenius matrix norm. For standard eigenvalues of $H=J_{2n}M$, the norma\-li\-zed residual is given by  
	$\|H \tilde{V}-\tilde{V}\tilde{\Lambda}\|_F/\|H \tilde{V}\|_F$, 
	where the columns of  
	$\tilde{V}\in\mathbb{C}^{2n\times 2k}$ are the computed eigenvectors of~$H$ associated with the  eigenvalues on the diagonal of 
	\mbox{$\tilde{\Lambda}=\diag(\tilde{\lambda}_1,\ldots,\tilde{\lambda}_{2k})$}.

\subsection{A matrix with known symplectic eigenvalues} 
We consider the spd matrix  
$M=Q\diag(D,D)Q^T$ with $D=\diag(1,\ldots,n)$ and $Q=KL(n/5,1.2,-\sqrt{n/5})$,
where $L(n/5,1.2,-\sqrt{n/5})\in\Spn$ is the symplectic Gauss transformation defined in \cite{Fass01}, and $K=\left[\begin{smallmatrix}\enskip\Re(U) & \Im(U)\\ -\Im(U) & \Re(U)\end{smallmatrix}\right]\in\OrS(2n)$ with unitary $U\in\mathbb{C}^{n\times n}$
produced by orthogonalizing a~randomly generated complex matrix. 
Then, the $k$ smallest symplectic eigenvalues of $M$ are $1,\ldots,k$. To exhibit the accuracy of computed symplectic eigenvalues $\tilde{d}_1,\ldots,\tilde{d}_k$, we calculate the 1-norm error defined as $\sum_{i=1}^k|\tilde{d}_i-i|$. In our tests, we choose $k=5$ and consider different values of $n$ in the range between 100 and 2000. The mentioned errors and the corresponding residuals for the three methods are shown in Figure \ref{fig:Exampl1_err}. The sought eigenvalues for $n=2000$ are given in Table~\ref{tab:Exampl1_eigs}. 
\begin{figure}[htpb] 
	\centering
		\begin{minipage}{.5\textwidth}
			\includegraphics[width=\textwidth]{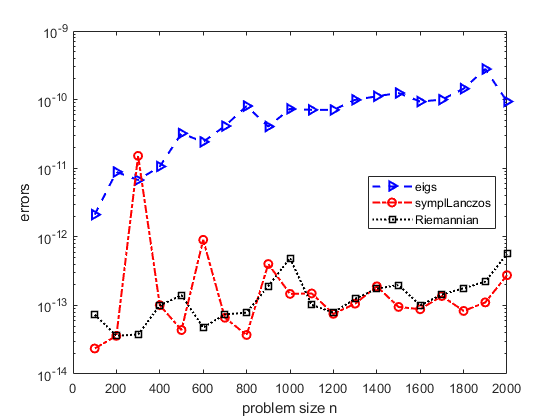}
		\end{minipage}
		\begin{minipage}{.5\textwidth}
			\includegraphics[width=\textwidth]{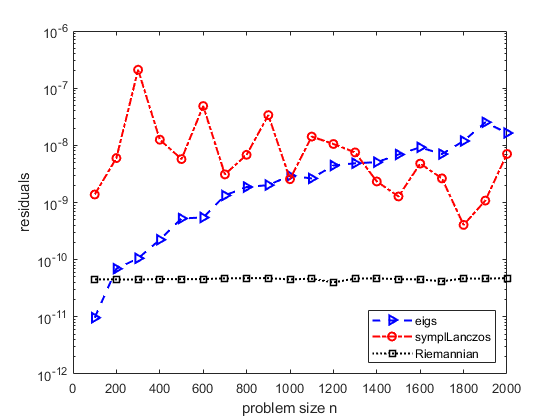} 
		\end{minipage}
	\caption{A matrix with known symplectic eigenvalues: the 1-norm errors of the computed symplectic eigenvalues (left) and the corresponding normalized residuals (right)}
	\label{fig:Exampl1_err}
\end{figure}

\begin{table}[tbhp]
	{\footnotesize
		\caption{5 smallest symplectic eigenvalues of a $4000\times 4000$ spd matrix $M$ computed by different methods}
		\label{tab:Exampl1_eigs}
		\begin{center}
			\begin{tabular}{c c c} \hline
				i$\times$\texttt{eigs}($H$) & symplLanczos($M$) & Riemannian($M$) \\ \hline
				\;0.000000000003296i + 1.000000000009247 & 1.000000000000058 & 1.000000000000008 \\
				-0.000000000022122i + 1.999999999995145  & 2.000000000000043 & 1.999999999999957 \\
				\;0.000000000015139i + 3.000000000002913 & 3.000000000000062 & 3.000000000000074 \\
				\;0.000000000023914i + 3.999999999977669 & 3.999999999999927 & 3.999999999999944  \\
				-0.000000000011256i + 4.999999999993021& 4.999999999999960 & 4.999999999999617 \\ \hline
			\end{tabular}
		\end{center}
	}
\end{table}

\subsection{Weakly damped gyroscopic systems}
In the stability analysis of gyroscopic systems, one needs to solve a~special quadratic eigenvalue problem (QEP) 
$(\lambda^2\mathcal{M} + \lambda\, \mathcal{G} + \mathcal{K})x = 0$, where $\mathcal{M}\in \SPD(n)$, $\mathcal{G}\in\skewset(n)$ and $\mathcal{K}\in\SPD(n)$ are, respectively, the mass, damping and stiffness matrices of the underlying mechanical structure. One can linearize this QEP and turn it into the standard EVP for the Hamiltonian matrix 
$$
H = \begin{bmatrix}
I&-\frac{1}{2}\mathcal{G}\\0& I
\end{bmatrix}
\begin{bmatrix}
0&-\mathcal{K}\\\mathcal{M}^{-1}& 0
\end{bmatrix}
\begin{bmatrix}
I&-\frac{1}{2}\mathcal{G}\\0& I
\end{bmatrix} = 
\begin{bmatrix}
-\frac{1}{2}\mathcal{G}\mathcal{M}^{-1} & 
\frac{1}{4}\mathcal{G}\mathcal{M}^{-1}\mathcal{G}-\mathcal{K}\\ 
\mathcal{M}^{-1} & -\frac{1}{2}\mathcal{M}^{-1}\mathcal{G}
\end{bmatrix},
$$
see \cite{BennFS08} for details. This leads to the fact that $J_{2n}^TH$
is symmetric negative definite if $\mathcal{G}$ is 
small enough. In our experiments, we use therefore the spd matrix 
$M=J_{2n}H$.

In the first test, we generate $\mathcal{M}$, $\mathcal{G}$ and $\mathcal{K}$ by an~eigenfunction discretization of a~wire saw model as described in \cite[Section~2]{WeiK00} with the wire speed 
$v = 0.0306$ and the dimension $n = 2000$ followed by a~scaling down of $\mathcal{G}$ by 1e-3. 
The eigenvalues computed by the three methods and the corresponding normalized residuals are given in Table~\ref{tab:Exampl2_eigs}. 

\begin{table}[tbhp]
	{\footnotesize
		\caption{5 smallest symplectic eigenvalues of a $4000\times 4000$ spd matrix $M=J_{2n}H$ generated from the wire saw model computed by different methods}
		\label{tab:Exampl2_eigs}
		\begin{center}
			\begin{tabular}{ l c c} \hline
\hspace*{25mm}i$\times$\texttt{eigs}($H$) & symplLanczos($M$) & Riemannian($M$) \\ \hline
\hspace*{7mm}0.000000000000002i +\, 3.140121476801627 & \;3.140121476801632 & \;3.140121476801794 \\			\hspace*{5mm} -0.000000000000001i +\, 6.280242953603250 & \;6.280242953603265 & \;6.280242953605164 \\
\hspace*{7mm}0.000000000000013i +\, 9.420364430404952 & \;9.420364430404895 &\, 9.420364430404506\\
\hspace*{7mm}0.000000000000037i +12.560485907206663	&12.560485907206548 & 12.560485907211794\\
\hspace*{5mm} -0.000000000000077i +15.700607384008093 &15.700607384008212 & 15.700607384223552\\ \hline
\!\!\!\!Residual: \hspace*{17mm} 1.4e-12 &1.7e-10 &1.3e-14 \\ \hline
			\end{tabular}
		\end{center}
	}
\end{table}

In the second test, we employ the data matrices $\mathcal{M}$ and $\mathcal{K}$
from a discretized model of a~piston rod inside a~combustion engine \cite{FehrGHFWE18}. This model has size 
$n = 8053$. Because matrix $\mathcal{G}$ in this model is not skew-symmetric, we replace it with a sparse  randomly generated skew-symmetric matrix whose pattern is the same as that of $\mathcal{M}$.
As the matrices in this model are large in magnitude, to improve the efficiency
of our method, we scale the matrix $H$ 
by a factor of 1e-5. The obtained results given in Table~\ref{tab:Exampl3_eigs} are for these scaled data.

\begin{table}[tbhp]
	{\footnotesize
		\caption{5 smallest symplectic eigenvalues of a $16106\times 16106$ spd matrix $M=J_{2n}H$ generated from the piston rod model computed by different methods}
		\label{tab:Exampl3_eigs}
		\begin{center}
			\begin{tabular}{ l c c} \hline
\hspace*{25mm}i$\times$\texttt{eigs}($H$) & symplLanczos($M$) & Riemannian($M$) \\ \hline
\hspace*{5mm} -0.000000000000001i + 0.162084145743768 & 0.162084145770035 & 0.162084145232661 \\
\hspace*{7mm}0.000000000000001i + 0.325674702254120 & 0.325674702270259 & 0.325674702005421 \\
\hspace*{7mm}0.000000000000006i + 0.663619676318176 & 0.663619676324319 & 0.663619676186475\\
\hspace*{7mm}0.000000000000001i + 1.350097974209022 & 1.350097974210526 & 1.350097974141396\\
\hspace*{5mm} -0.000000000000004i + 2.173559065028063 & 2.173559065366786 & 2.173559064987688\\ \hline
\!\!\!\! Residual: \hspace*{17mm} 4.4e-10 &7.6e-7 &9.9e-12 \\ \hline
			\end{tabular}
		\end{center}
	}
\end{table}

Some observations and remarks can be stated from these numerical examples. The comparisons might be a bit biased since \texttt{eigs} is not designed for structured matrices, whereas the symplectic Lanczos method and the Riemannian optimization method exploit the structure of the EVP. This explains why in all three test examples the eigenvalues computed by \texttt{eigs}($H$) are not purely imaginary.
Though, in the symplectic Lanczos method, the residuals, which also depend on the accuracy of the symplectic eigenvectors, are not as small as expected, the first example shows that this method produces good approximations to symplectic eigenvalues. Compared to that, our method yields satisfying results in the sense that both errors and residuals are small. It should however be noted that slow convergence, especially near minimizers, was sometimes experienced in our tests. This is well-known for first-order optimization methods and poses a~need for development of second-order methods.

\section{Conclusion}\label{Sec:Concl}
We have established various theoretical properties for the symplectic eigenvalue trace minimization problem. Many of them are symplectic extensions of known results for the standard problem. We have also proposed a Riemannian optimization-based numerical method that resorts to a recent development about optimization on the symplectic Stiefel manifold. This method can also be employed to compute standard eigenvalues 
of positive-definite Hamiltonian matrices. Nume\-ri\-cal examples demonstrate that the proposed method is
comparable to existing approaches in the sense of accuracy. 

\section*{Acknowledgments}
We would like to thank B.~Fr\"{o}hlich for providing us with the data for the piston rod model.

\appendix 
\section{Alternative proof of the necessity in Theorem \ref{prop:minimizers_Orthosympl}(ii)} \label{app:Anotherway} Theorem~\ref{prop:Critical_Orthosympl} is so strong that it does not only characterize the set of the critical points of the {minimization problem \eqref{eq:SymplOptimProb} but also helps to obtain the set of the global minimizers as clarified in Theorem~\ref{prop:minimizers_Orthosympl}(ii). In this extra section, we will present another proof of this theorem which
does not resort to Theorem~\ref{prop:Critical_Orthosympl} and its consequences. 

Let $X_*\in\Spkn$ be a minimizer of \eqref{eq:SymplOptimProb}. Then it satisfies the KKT condition \eqref{eq:KKTcond} or, equivalently, $	X_*^TMX_* = J_{2k}L_*$.
	Since $	X_*^TMX_*$ is spd,	an~application of Williamson's theorem implies the existence of 
	$K \in \Spk$ such that
	\begin{equation}\label{eq:lemproof4_2}
	K^TJ_{2k}L_*K = K^TX_*^TMX_*K=
	\begin{bmatrix} \Lambda & 0 \\ 0 & \Lambda\end{bmatrix}
	\end{equation}
	with $\Lambda=\diag(\lambda_1,\ldots,\lambda_k)$.
	
Next, we show that 
$\lambda_j=d_j$, $j=1,\ldots,k$. To this end, let us add more columns to $X_*$ to make 
	$\tilde{X}_* \in Sp(2n)$ such that its $1$-st, $\ldots$, $k$-th, $(n+1)$-th, $\ldots$, $(n+k)$-th columns are those of $X_*$, see \cite[Theorem~1.15]{deGo06}. It was shown in \cite[Proposition~8.14]{deGo06}, that the symplectic spectrum is symplectic invariant. 
	This yields that the symplectic eigenvalues of  $\tilde{X}_*^TM\tilde{X}_*$ are still $d_j, j = 1,\ldots,n$. Moreover, $X_*^TMX_*^{}$ is the so-called s-principal $2k\times 2k$ submatrix of $\tilde{X}_*^TM\tilde{X}_*^{}$, i.e., $X_*^TMX_*^{}$ is obtained from
	$\tilde{X}_*^TM\tilde{X}_*^{}$ by deleting its row and columns with the indices $k+1,\ldots, n, 
	n+k+1,\ldots, 2n$.
	From the symplectic analog of Cauchy's interlacing theorem \cite{KrbeTV14,BhatJ15}, we deduce that
	$d_j \leq \lambda_j$ for $j = 1,\ldots,k$.
	On the other hand, taking into account that
	$X_*$ is a global minimizer of \eqref{eq:SymplOptimProb}, we obtain
	$$
	2\sum_{j=1}^k \lambda_j=\tr(K^TJ_{2k}L_*K)=\min_{Y\in\Spk}\mathrm{tr}(Y^TX_*^TMX_*^{}Y)\leq \mathrm{tr}(X_*^TMX_*^{})
	= 2\sum_{j=1}^k d_j,
	$$
	and, hence,
$\lambda_j = d_j$ for $j=1,\ldots,k$.	
	Further, it follows from \eqref{eq:KKTcond} and \eqref{eq:lemproof4_2} that
	\begin{equation*}
	MX_*K = J_{2n}X_*KK^{-1}L_*	K=
	J_{2n}X_*K\begin{bmatrix}0&- \Lambda\\ \Lambda &0\end{bmatrix}.
	\end{equation*}
This implies that the columns of $X_{1:k}:=X_*K$ form a normalized symplectic eigenvector set associated with the symplectic eigenvalues $d_1,\ldots,d_k$. 

It remains to show that $K\in\OrS(2k)$. Define $F = K^{-1}$. Since $X_* = X_{1:k}F$ is a~global minimizer of \eqref{eq:SymplOptimProb}, it follows that
\begin{equation}\label{eq:theoproof1_1}
2\sum_{i=1}^k d_i = \tr (X_*^TMX_*)
= \tr (F^TX_{1:k}^TMX_{1:k}^{}F)	 = \tr (F^T\mathrm{diag}(\Lambda,\Lambda)F).
\end{equation}
We now express $F$ in the block form as
$
F = \left[\begin{smallmatrix}
A&B\\C&G
\end{smallmatrix}\right].
$ 
By Proposition~\ref{prop:SymplGroup}{(i)}, we have $F^T\in \Spk$. This	results in the following constraints for the submatrices
\begin{equation}\label{eq:SymplSubmatConstr}
AG^T - CB^T =I,\quad AB^T = BA^T,\quad CG^T = GC^T.
\end{equation}
Then the right-hand side of \eqref{eq:theoproof1_1} can be more detailed as
\begin{align*}
2\sum_{i=1}^{k} d_i &= \tr \bigl(F^T\mathrm{diag}(\Lambda,\Lambda)F\bigr)
= \tr (A^T\Lambda\,A + C^T\Lambda\,C + B^T\Lambda\,B + G^T\Lambda\,G)\\
&=\sum_{i=1}^k d_i\sum_{j=1}^{k}(a_{ij}^2 + g_{ij}^2 + c_{ij}^2 + b_{ij}^2)
\geq 2\sum_{i=1}^{k} d_i\sum_{j=1}^{k}(a_{ij}g_{ij} - c_{ij}b_{ij}) =2\sum_{i=1}^{k} d_i,
\end{align*}
where ``$\geq$" appears due to the facts that $(a_{ij}-g_{ij})^2 \geq 0$ and $(c_{ij}+b_{ij})^2 \geq 0$ for all $i,j=1,\ldots,k$, and the last equality follows from the first relation in \eqref{eq:SymplSubmatConstr}. The equality case happens if and only if 
$a_{ij}= g_{ij}$ and $c_{ij} = -b_{ij}$ for $i, j = 1,\ldots,k$.
Thus, $A=G$ and $C=-B$. Then by Proposition~\ref{prop:SymplGroup}{(ii)}, we obtain that $F\in \OrS(2k)$ and, hence, $K = F^{-1}\in \OrS(2k)$.
\hfill\proofbox

The last part of this proof is based on the ideas in \cite[Theorems~5, 6]{BhatJ15}. It is however more direct and does not invoke the notions of doubly stochastic and doubly superstochastic matrices. 

\section{Proof of Lemma~\ref{lem:TangentNull}}
\label{app:proofLemTangentNull}
We show the existence of $X_*$ only, as the proof for $X^*$ is similar. By Corollaries~\ref{corol:CriticalClasses} and \ref{corol:MimimizersClasses}, we can replace $X_0$ and $X_*$ by $X_{\mathcal{I}_k}K_0$ and $X_{1:k}K_*$, respectively, with some $K_0, K_* \in \OrS(2k)$ and $\mathcal{I}_k \subset \{1,\ldots,n\}$. Let us assume that this lemma holds for the  critical point $X_{\mathcal{I}_k}$, i.e., there exists a~global minimizer $X_*$ of \eqref{eq:SymplOptimProb} satisfying  
\begin{equation}\label{eq:ProofTangentNull_1}
X_*^TJ_{2n}X_{\mathcal{I}_k} + X_{\mathcal{I}_k}^TJ_{2n}X_* = 0.
\end{equation}
Then we have
\begin{align*}
(X_*K_0)^TJ_{2n}X_0 + X_0^TJ_{2n}(X_*K_0) &= K_0^TX_*^TJ_{2n}X_{\mathcal{I}_k}K_0 + K_0^TX_{\mathcal{I}_k}^TJ_{2n}X_*K_0\\
&= K_0^T(X_*^TJ_{2n}X_{\mathcal{I}_k} + X_{\mathcal{I}_k}^TJ_{2n}X_*)K_0 =0.
\end{align*}
This means that
$X_*K_0$ is the sought global minimizer corresponding to $X_0=X_{\mathcal{I}_k}K_0$. 

We now prove the above assumption. Our goal is to construct $K_*\in \OrS(2k)$ such that
	$X_*=X_{1:k}K_*$ satisfies \eqref{eq:ProofTangentNull_1} and is the global minimizer of \eqref{eq:SymplOptimProb}. 
Let $O = X_{1:k}^TJ_{2n}X_{\mathcal{I}_k}$ for $\mathcal{I}_k=\{i_1,\ldots, i_k\} \subset \{1,\ldots,n\}$. We can see that $O$ can be written in the block form as
\begin{equation*}
O = \begin{bmatrix}
0 & O_1\\-O_1& 0
\end{bmatrix}\in \Rbkk,
\end{equation*} 
where $O_1 = [x_1,\dots,x_k]^TJ_{2n}[x_{n+i_1},\dots,x_{n+i_k}] \in \mathbb{R}^{k\times k}$. 
Let $c$ denote the number of common indices $\{1,\dots,k\} \cap \mathcal{I}_k$ with $0 \leq c \leq k$. Taking Proposition~\ref{prop:SymplGroup}(ii) into account, we are searching for $K_*\in \OrS(2k)$ of the form
$$
K_* = \begin{bmatrix}
\enskip K_1&K_2\\-K_2&K_1
\end{bmatrix},
$$
where $K_1$ and $K_2$ satisfy 
\begin{align}
K_1^TK_2 & = K_2^TK_1, \qquad K_1^TK_1 + K_2^TK_2 = I, \label{eq:cond1} \\
K_1^TO_1 & = - O_1^TK_1,\quad\, K_2^TO_1 = O_1^TK_2. \label{eq:cond2}
\end{align}
The conditions \eqref{eq:cond1} guarantee the orthosymplecticity of $K_*$, whereas the conditions \eqref{eq:cond2} imply \eqref{eq:ProofTangentNull_1}. By definition, $O_1$ contains exactly $c$ 1's. Let us denote their positions by $(i_1,j_1),\dots,(i_c,j_c)$. We moreover choose other $k-c$ positions $(i_{c+1},j_{c+1}),\dots,(i_{k},j_k)$ in such a~way that 
if we put 1 in $O_1$ at all these positions, then the resulting matrix becomes a~permutation of the identity. 
Let us note that while the set $(i_1,j_1),\dots,(i_c,j_c)$ is fixed upon the given matrix $O_1$, there are multiple choices for $(i_{c+1},j_{c+1}),\dots,(i_{k},j_k)$. We will construct $K_*$ as follows: 
\begin{equation*}
(K_1)_{ij} = \begin{cases}
\cos\phi_l, &\mbox{if } (i,j,l) \in \{(i_{c+1},j_{c+1},c+1),\dots,(i_{k},j_k,k)\},\\
0, &\mbox{otherwise},
\end{cases}
\end{equation*}
\begin{equation*}
(K_2)_{ij} = \begin{cases}
1,&\mbox{if } (i,j) \in \{(i_{1},j_{1}),\dots,(i_{c},j_c)\},\\
\sin\phi_l, &\mbox{if } (i,j,l) \in \{(i_{c+1},j_{c+1},c+1),\dots,(i_{k},j_k,k)\},\\
0, &\mbox{otherwise},
\end{cases}
\end{equation*}
with $\phi_l \in \mathbb{R}$. 
Note that we can use $-1$ instead of $1$ in $K_2$. One directly verifies that
	\begin{align*}
	 (K_1^TK_2)_{ij} &= (K_2^TK_1)_{ij}\\
	&= \begin{cases}
	\cos\phi_l\sin\phi_l,&\mbox{if } (i,j,l) \in \{(c+1,c+1,c+1),\dots,(k,k,k)\},\\
	0, &\mbox{otherwise},
	\end{cases} \\
	(K_1^TK_1)_{ij} & = \begin{cases}
	\cos^2\phi_l,&\mbox{if } (i,j,l) \in \{(c+1,c+1,c+1),\dots,(k,k,k)\},\\
	0, &\mbox{otherwise},
	\end{cases} \\
	(K_2^TK_2)_{ij} & = \begin{cases}
	1,&\mbox{if } (i,j) \in \{(1,1),\dots,(c,c)\},\\
	\sin^2\phi_l, &\mbox{if } (i,j,l) \in \{(c+1,c+1,c+1),\dots,(k,k,k)\},\\
	0, &\mbox{otherwise}.
	\end{cases}
	\end{align*}
	and, hence, the relations in \eqref{eq:cond1} are satisfied. Furthermore, we have
	\begin{align*} 
	(K_1^TO_1)_{ij} & = (-O_1^TK_1)_{ij} = 0, \quad i,j=1,\ldots,k, \\
	(K_2^TO_1)_{ij} & = (O_1^TK_2)_{ij} 
	= \begin{cases}
	1,&\mbox{if } (i,j) \in \{(1,1),\dots,(c,c)\},\\
	0, &\mbox{otherwise}
	\end{cases}
	\end{align*}
implying the relations in \eqref{eq:cond2}. 
\hfill\proofbox

Though covered in the proof, we still want to show two special cases of $c$. If $c=0$, then $O = 0$ and, hence, we can choose any $K_* \in \OrS(2k)$. If $c=k$, i.e., $X_0$ is a~minimizer, then $O_1 = I$. In this case, we can take, for example, $K_*=J_{2k}$.

\bibliographystyle{siamplain}
\bibliography{references}

\begin{thebibliography}{10}

\bibitem{AbsiBG06}
{\sc P.-A. Absil, C.~Baker, and K.~Gallivan}, {\em A truncated-{CG} style
  method for symmetric generalized eigenvalue problems}, Journal of
  Computational and Applied Mathematics, 189 (2006), pp.~274--285,
  \url{https://doi.org/10.1016/j.cam.2005.10.006}.

\bibitem{AbsiBGS05}
{\sc P.-A. Absil, C.~Baker, K.~Gallivan, and A.~Sameh}, {\em Adaptive model
  trust region methods for generalized eigenvalue problems}, in Computational
  Science - ICCS 2005, Springer, 2005, pp.~33--41,
  \url{https://doi.org/10.1007/11428831_5}.

\bibitem{AbsiMS04}
{\sc P.-A. Absil, R.~Mahony, and R.~Sepulchre}, {\em {R}iemannian geometry of
  {G}rassmann manifolds with a~view on algorithmic computations}, Acta Appl.
  Math, 8 (2004), pp.~199--220, \url{https://doi.org/10.1023/B:ACAP.0000}.

\bibitem{AbsiMS08}
{\sc P.-A. Absil, R.~Mahony, and R.~Sepulchre}, {\em Optimization Algorithms on
  Matrix Manifolds}, Princeton University Press, Princeton, NJ, 2008.

\bibitem{Amod03}
{\sc P.~Amodio}, {\em A symplectic {L}anczos-type algorithm to compute the
  eigenvalues of positive definite {H}amiltonian matrices}, in Computational
  Science—ICCS 2003, P.~Sloot, D.~Abramson, A.~Bogdanov, J.~Dongarra,
  A.~Zomaya, and Y.~Gorbachev, eds., Lecture Notes in Computer Science, vol.
  2657 (part II), Springer-Verlag Berlin, 2003, pp.~139--148.

\bibitem{Amod06}
{\sc P.~Amodio}, {\em On the computation of few eigenvalues of positive
  definite {H}amiltonian matrices}, Future Generation Computer Systems, 22
  (2006), pp.~403--411, \url{https://doi.org/10.1016/j.future.2004.11.027}.

\bibitem{AmodIT05}
{\sc P.~Amodio, F.~Iavernaro, and D.~Trigiante}, {\em Conservative
  perturbations of positive definite {H}amiltonian matrices}, Numer. Linear
  Algebra Appl, 12 (2005), pp.~117--125, \url{https://doi.org/10.1002/nla.409}.

\bibitem{BaiL12}
{\sc Z.~Bai and R.-C. Li}, {\em Minimization principles for the linear response
  eigenvalue problem {I}: Theory}, SIAM J. Matrix Anal. Appl., 33 (2012),
  p.~1075–1100, \url{https://doi.org/10.1137/110838960}.

\bibitem{BaiL13}
{\sc Z.~Bai and R.-C. Li}, {\em Minimization principles for the linear response
  eigenvalue problem {II}: Computation}, SIAM J. Matrix Anal. Appl., 34 (2013),
  p.~392–416, \url{https://doi.org/10.1137/110838972}.

\bibitem{BaiL14}
{\sc Z.~Bai and R.-C. Li}, {\em Minimization principles and computation for the
  generalized linear response eigenvalue problem}, BIT Numer. Math., 54 (2014),
  pp.~31--54, \url{https://doi.org/10.1007/s10543-014-0472-6}.

\bibitem{BakeAG06}
{\sc C.~Baker, P.-A. Absil, and K.~Gallivan}, {\em An implicit {R}iemannian
  trust-region method for the symmetric generalized eigenproblem}, in
  Computational Science - ICCS 2006, Springer, 2006, pp.~210--217.

\bibitem{BennF97}
{\sc P.~Benner and H.~Fassbender}, {\em An implicitly restarted symplectic
  {L}anczos method for the {H}amiltonian eigenvalue problem}, Linear Algebra
  Appl., 263 (1997), pp.~75--111,
  \url{https://doi.org/10.1016/S0024-3795(96)00524-1}.

\bibitem{BennF98}
{\sc P.~Benner and H.~Fassbender}, {\em The symplectic eigenvalue problem, the
  butterfly form, the {SR} algorithm, and the {L}anczos method}, Linear Algebra
  Appl., 275-276 (1998), pp.~19--47,
  \url{https://doi.org/10.1016/S0024-3795(97)10049-0}.

\bibitem{BennFS08}
{\sc P.~Benner, H.~Fassbender, and M.~Stoll}, {\em Solving large-scale
  quadratic eigenvalue problems with {H}amiltonian eigenstructure using a
  structure-preserving {K}rylov subspace method}, ETNA, 29 (2008),
  pp.~212--229.

\bibitem{BennFS11}
{\sc P.~Benner, H.~Fa{\ss}bender, and M.~Stoll}, {\em A {H}amiltonian
  {K}rylov–{S}chur-type method based on the symplectic {L}anczos process},
  Linear Algebra Appl., 435 (2011), pp.~578--600,
  \url{https://doi.org/10.1016/j.laa.2010.04.048}.

\bibitem{BennKM05}
{\sc P.~Benner, D.~Kressner, and V.~Mehrmann}, {\em Skew-{H}amiltonian and
  {H}amiltonian eigenvalue problems: Theory, algorithms and applications}, in
  Proceedings of the Conference on Applied Mathematics and Scientific
  Computing, 2005, pp.~3--39, \url{https://doi.org/10.1007/1-4020-3197-1_1}.

\bibitem{BhatJ15}
{\sc R.~Bhatia and T.~Jain}, {\em On the symplectic eigenvalues of positive
  definite matrices}, J. Math. Phys., 56 (2015), p.~112201,
  \url{https://doi.org/10.1063/1.4935852}.

\bibitem{BirtCC20}
{\sc P.~Birtea, I.~Ca\c{s}u, and D.~Com\v{a}nescu}, {\em Optimization on the
  symplectic group}, Monatshefte Math.,  (2020),
  \url{https://doi.org/10.1007/s00605-020-01369-9}.

\bibitem{BunsM86}
{\sc A.~Bunse-Gerstner and V.~Mehrmann}, {\em A symplectic {QR} like algorithm
  for the solution of the real algebraic {R}iccati equation}, IEEE Trans.
  Automat. Control, 31 (1986), pp.~1104 -- 1113,
  \url{https://doi.org/10.1109/TAC.1986.1104186}.

\bibitem{deGo06}
{\sc M.~de~Gosson}, {\em Symplectic Geometry and Quantum Mechanics}, Advances
  in Partial Differential Equations, Birkh\"auser, Basel, 2006.

\bibitem{EdelAS98}
{\sc A.~Edelman, T.~Arias, and S.~Smith}, {\em The geometry of algorithms with
  orthogonality constraints}, {SIAM} J. Matrix Anal. Appl., 20 (1998),
  pp.~303--353, \url{https://doi.org/10.1137/S0895479895290954}.

\bibitem{EiseTRS08}
{\sc J.~Eisert, T.~Tyc, T.~Rudolph, and B.~Sanders}, {\em Gaussian quantum
  marginal problem}, Commun. Math. Phys., 280 (2008), pp.~263--280,
  \url{https://doi.org/10.1007/s00220-008-0442-4}.

\bibitem{Fass01}
{\sc H.~Fassbender}, {\em The parameterized {SR} algorithm for symplectic
  (butterfly) matrices}, Mathematics of Computation, 70 (2000), pp.~1515--1541,
  \url{https://doi.org/10.1090/S0025-5718-00-01265-5}.

\bibitem{Fass02}
{\sc H.~Fassbender}, {\em Symplectic Methods for the Symplectic Eigenproblem},
  Springer US, Philadelphia; PWN-Polish Scientific, 2002.

\bibitem{FehrGHFWE18}
{\sc J.~Fehr, D.~Grunert, P.~Holzwarth, B.~Fr\"{o}hlich, N.~Walker, and
  P.~Eberhard}, {\em {MOREMBS} -- {A} model order reduction package for elastic
  multibody systems and beyond}, in Reduced-Order Modeling (ROM) for Simulation
  and Optimization, W.~Keiper, A.~Milde, and S.~Volkwein, eds., 2018,
  pp.~141--166, \url{https://doi.org/10.1007/978-3-319-75319-5_7}.

\bibitem{Fior16}
{\sc S.~Fiori}, {\em A {R}iemannian steepest descent approach over the
  inhomogeneous symplectic group: Application to the averaging of linear
  optical systems}, Appl. Math. Comput., 283 (2016), pp.~251--264,
  \url{https://doi.org/10.1016/j.amc.2016.02.018}.

\bibitem{Fome95}
{\sc A.~Fomenko}, {\em Symplectic Geometry}, vol.~5 of Advanced Studies in
  Contemporary Mathematics, Gordon and Breach Science Publishers, Amsterdam,
  1995.

\bibitem{Fran87}
{\sc B.~Francis}, {\em A Course in {$H_\infty$} Control Theory}, vol.~88 of
  Lecture Notes in Control and Information Science, Springer, Heidelberg, 1987.

\bibitem{GSAS20}
{\sc B.~Gao, N.~Son, P.-A. Absil, and T.~Stykel}, {\em Riemannian optimization
  on the symplectic {S}tiefel manifold}, Preprint UCL-INMA-2020.04, UCLouvain,
  Louvain-la-Neuve, June 2020.

\bibitem{GoluV13}
{\sc G.~Golub and C.~V. Loan}, {\em Matrix Computations. 4th ed}, The Johns
  Hopkins University Press, Baltimore, London, 2013.

\bibitem{High08}
{\sc N.~Higham}, {\em Functions of Matrices: Theory and Computation}, {SIAM},
  Philadelphia, PA, 2008, \url{https://doi.org/10.1137/1.9780898717778}.

\bibitem{HinrS91}
{\sc D.~Hinrichsen and N.~Son}, {\em Stability radii of linear discrete-time
  systems and symplectic pencils}, Int. J. Robust Nonlinear Control, 1 (1991),
  pp.~79--97, \url{https://doi.org/10.1002/rnc.4590010204}.

\bibitem{Hiro06}
{\sc T.~Hiroshima}, {\em Additivity and multiplicativity properties of some
  {G}aussian channels for gaussian inputs}, Phys. Rev. A, 73 (2006), p.~012330,
  \url{https://doi.org/10.1103/PhysRevA.73.012330}.

\bibitem{HofeZ11}
{\sc H.~Hofer and E.~Zehnder}, {\em Symplectic Invariants and Hamiltonian
  Dynamics}, Birkh\"auser, Basel, 2011,
  \url{https://doi.org/10.1007/978-3-0348-0104-1}.

\bibitem{IdelGW17}
{\sc M.~Idel, S.~Gaona, and M.~Wolf}, {\em Perturbation bounds for
  {W}illiamson's symplectic normal form}, Linear Algebra Appl., 525 (2017),
  pp.~45--58, \url{https://doi.org/10.1016/j.laa.2017.03.013}.

\bibitem{Ikra91}
{\sc K.~D. Ikramov}, {\em The conditions for the reducibility and canonical
  forms of {H}amiltonian matrices with pure imaginary eigenvalues}, Zh.
  Vychisl. Mat. Mat. Fiz., 31 (1991), pp.~1123--1130.

\bibitem{Ikra18}
{\sc K.~D. Ikramov}, {\em On the symplectic eigenvalues of positive definite
  matrices}, Moscow University Computational Mathematics and Cybernetics, 42
  (2018), pp.~1--4, \url{https://doi.org/10.3103/S0278641918010041}.

\bibitem{JainM20}
{\sc T.~Jain and H.~Mishra}, {\em Derivatives of symplectic eigenvalues and a
  {L}idskii type theorem}, Canad. J. Math.,  (2020),
  \url{https://doi.org/10.4153/S0008414X2000084X}.

\bibitem{KovaV95}
{\sc J.~Kova\v{c}-Striko and K.~Veseli\'{c}}, {\em Trace minimization and
  definiteness of symmetric pencils}, Linear Algebra Appl., 216 (1995),
  pp.~139--158, \url{https://doi.org/10.1016/0024-3795(93)00126-K}.

\bibitem{KrbeTV14}
{\sc M.~Krbek, T.~Tyc, and J.~Vlach}, {\em Inequalities for quantum marginal
  problems with continuous variables}, J. Math. Phys., 55 (2014), p.~062201,
  \url{https://doi.org/10.1063/1.4880198}.

\bibitem{Kres05}
{\sc D.~Kressner}, {\em Numerical Methods for General and Structured Eigenvalue
  Problems}, Lecture Notes in Computational Science and Engineering, 46,
  Springer-Verlag, Berlin Heidelberg, 2005,
  \url{https://doi.org/10.1007/3-540-28502-4}.

\bibitem{LancR95}
{\sc P.~Lancaster and L.~Rodman}, {\em The Algebraic {R}iccati Equation},
  Oxford University Press, Oxford, 1995.

\bibitem{LeePPSJR19}
{\sc J.~Lee, I.~Panageas, G.~Piliouras, M.~Simchowitz, M.~Jordan, and
  B.~Recht}, {\em First-order methods almost always avoid strict saddle
  points}, Math. Program., 176 (2019), pp.~311--337,
  \url{https://doi.org/110.1007/s10107-019-01374-3}.

\bibitem{LianLB13}
{\sc X.~Liang, R.-C. Li, and Z.~Bai}, {\em Trace minimization principles for
  positive semi-definite pencils}, Linear Algebra Appl., 438 (2013),
  pp.~3085--3106, \url{https://doi.org/10.1016/j.laa.2012.12.003}.

\bibitem{VanL84}
{\sc C.~V. Loan}, {\em A symplectic method for approximating all the
  eigenvalues of a {H}amiltonian matrix}, Linear Algebra Appl., 61 (1984),
  pp.~233--251, \url{https://doi.org/10.1016/0024-3795(84)90034-X}.

\bibitem{NoceW06}
{\sc J.~Nocedal and S.~Wright}, {\em Numerical Optimization}, Springer Series
  in Operation Research and Finacial Engineering, Springer, Berlin/New York,
  2006.

\bibitem{Part13}
{\sc K.~Parthasarathy}, {\em The symplectry group of {G}aussian states in
  {$L^2(\mathbb{R}^n)$}}, in Prokhorov and Contemporary Probability Theory,
  Springer Proceedings in Mathematics and Statistics 33, Berlin Heidelberg,
  2013, pp.~349--369, \url{https://doi.org/10.1007/978-3-642-33549-5_21}.

\bibitem{PengM16}
{\sc L.~Peng and K.~Mohseni}, {\em Symplectic model reduction of {H}amiltonian
  systems}, SIAM J. Sci. Comput., 38 (2016), pp.~A1--A27,
  \url{https://doi.org/10.1137/140978922}.

\bibitem{Saad11}
{\sc Y.~Saad}, {\em Numerical Methods for Large Eigenvalue Problems}, SIAM,
  Philadelphia, 2011.

\bibitem{SameT00}
{\sc A.~Sameh and Z.~Tong}, {\em The trace minimization method for the
  symmetric generalized eigenvalue problem}, J. Comput. Appl. Math., 123
  (2000), pp.~155--175, \url{https://doi.org/10.1016/S0377-0427(00)00391-5}.

\bibitem{SameW82}
{\sc A.~Sameh and J.~Wisniewski}, {\em A trace minimization algorithm for the
  generalized eigenvalue problem}, SIAM J. Numer. Anal., 19 (1982),
  p.~1243–1259, \url{https://doi.org/10.1137/0719089}.

\bibitem{SimoCS99}
{\sc R.~Simon, S.~Chaturvedi, and V.~Srinivasan}, {\em Congruences and
  canonical forms for a positive matrix: {A}pplication to the
  {S}chweinler{--}{W}igner extremum principle}, J. Maths. Phys., 40 (1999),
  pp.~3632--3642, \url{https://doi.org/10.1063/1.532913}.

\bibitem{VanDSJ14}
{\sc A.~van~der Schaft and D.~Jeltsema}, {\em Port-{H}amiltonian systems
  theory: An introductory overview}, Foundations and Trends in Systems and
  Control, 1 (2014), pp.~173--378, \url{https://doi.org/10.1561/2600000002}.

\bibitem{WardGray78}
{\sc R.~Ward and L.~Gray}, {\em Eigensystem computation for skew-symmetric
  matrices and a class of symmetric matrices}, {ACM} Trans. Math. Software, 4
  (1978), pp.~278--285.

\bibitem{Watk04}
{\sc D.~Watkins}, {\em On {H}amiltonian and symplectic {L}anczos processes},
  Linear Algebra Appl., 385 (2004), pp.~23--45,
  \url{https://doi.org/10.1016/j.laa.2002.11.001}.
\newblock Special Issue in honor of Peter Lancaster.

\bibitem{Watk07}
{\sc D.~Watkins}, {\em The Matrix Eigenvalue Problem}, SIAM, Philadelphia, PA,
  2007, \url{https://doi.org/10.1137/1.9780898717808}.

\bibitem{WeiK00}
{\sc S.~Wei and I.~Kao}, {\em Vibration analysis of wire and frequency response
  in the modern wiresaw manufacturing process}, Journal of Sound and Vibration,
  231 (2000), pp.~1383--1395, \url{https://doi.org/10.1006/jsvi.1999.2471}.

\bibitem{Will36}
{\sc J.~Williamson}, {\em On the algebraic problem concerning the normal forms
  of linear dynamical systems}, Am. J. Math., 58 (1936), pp.~141--163.

\end{thebibliography}
\end{document}